\documentclass[a4paper,12pt,oneside,reqno]{amsart}
\usepackage{amsfonts,amsmath, amssymb,amsthm,amscd,mathtools}
\usepackage[usenames,dvipsnames]{color}
\usepackage{bm}
\usepackage{mathrsfs}
\usepackage[mpexclude,DIV13]{typearea}
\usepackage{verbatim}
\usepackage{hyperref}
\usepackage{graphicx}
\usepackage{multicol}
\usepackage{latexsym}
\usepackage{lscape}
\usepackage{epsfig}

\usepackage{hyperref}

\usepackage[margin=1.1in]{geometry}

\newcommand{\ii}{\mathbf i}

\newcommand{\sgn}{{\rm sgn}}

\renewcommand{\P}{\mathcal P}

\renewcommand{\Re}{\operatorname{Re}}

\newcommand{\eps}{\varepsilon}

\renewcommand{\S}{\mathfrak S}

\newcommand{\x}{\mathbf{X}}
\newcommand{\y}{\mathbf{Y}}
\newcommand{\z}{\mathbf{Z}}

\newcommand{\W}{\mathbb{W}}
\renewcommand{\H}{\mathcal H}

\newcommand{\T}{\mathcal T}

\newcommand{\C}{\mathbb{C}}
\newcommand{\Z}{\mathbb{Z}}

\newcommand{\R}{\mathbb{R}}
\newcommand{\PP}{\mathbb{P}}
\newcommand{\EE}{\mathbb{E}}

\newcommand{\Cr}{\mathcal {S}_\tau}

\newtheorem{theorem}{Theorem}[section]

\newtheorem*{theorem*}{Theorem}
\newtheorem{lemma}[theorem]{Lemma}
\newtheorem{definition}[theorem]{Definition}
\newtheorem{proposition}[theorem]{Proposition}

\newtheorem{corollary}[theorem]{Corollary}

\begin{document}
\begin{abstract}
 We study the asymmetric six-vertex  model in the quadrant with parameters on the stochastic
line. We show that the random height function of the model converges to an explicit deterministic
limit shape as the mesh size tends to $0$. We further prove that the one-point fluctuations around
the limit shape are asymptotically governed by the GUE Tracy--Widom distribution. We also explain
an equivalent formulation of our model as an interacting particle system, which can be viewed as a
discrete time generalization of ASEP started from the step initial condition. Our results confirm
an earlier prediction of Gwa and Spohn (1992) that this system belongs to the KPZ universality
class.
\end{abstract}

\title{Stochastic six-vertex model}
\author[A. Borodin]{Alexei Borodin}
\address{A. Borodin,
Massachusetts Institute of Technology,
Department of Mathematics,
77 Massachusetts Avenue, Cambridge, MA 02139-4307, USA, and
Institute for Information Transmission Problems, Bolshoy Karetny per. 19, Moscow 127994, Russia}
\email{borodin@math.mit.edu}

\author[I. Corwin]{Ivan Corwin}
\address{I. Corwin, Columbia University,
Department of Mathematics, 2990 Broadway, New York, NY 10027, USA, and Clay Mathematics Institute,
10 Memorial Blvd. Suite 902, Providence, RI 02903, USA, and Massachusetts Institute of Technology,
Department of Mathematics, 77 Massachusetts Avenue, Cambridge, MA 02139-4307, USA, and Institute
Henri Poincare, 11 Rue Pierre et Marie Curie, 75005 Paris, France.} \email{ivan.corwin@gmail.com}

\author[V. Gorin]{Vadim Gorin}
\address{V. Gorin,
Massachusetts Institute of Technology,
Department of Mathematics,
77 Massachusetts Avenue, Cambridge, MA 02139-4307, USA, and Institute for Information Transmission Problems, Bolshoy Karetny per. 19, Moscow 127994, Russia}
\email{vadicgor@gmail.com}

\maketitle

\setcounter{tocdepth}{3}
\tableofcontents
\hypersetup{linktocpage}

\section{Introduction}

In this article we study a stochastic system at the interface of equilibrium lattice models and
non-equilibrium interacting particle systems.

From the point of view of lattice models, we deal with the six--vertex (or ``square--ice'') model.
The configurations of the six--vertex model are assignments of one of 6 types of $H_2O$ molecules
shown in Figure \ref{Figure_six} to the vertices of (a subdomain of) square grid in such a way
that the $O$ atoms are at the vertices of the grid. To each $O$ atom there are two $H$ atoms
attached, so that they are at angles $90^\circ$ or $180^\circ$ to each other, along the grid
lines, and between any two adjacent $O$ atoms there is exactly one $H$. Figure
\ref{Figure_configuration_H2O} shows an example of a configuration.

\begin{figure}[h]
\begin{center}
  {\scalebox{0.6}{\includegraphics{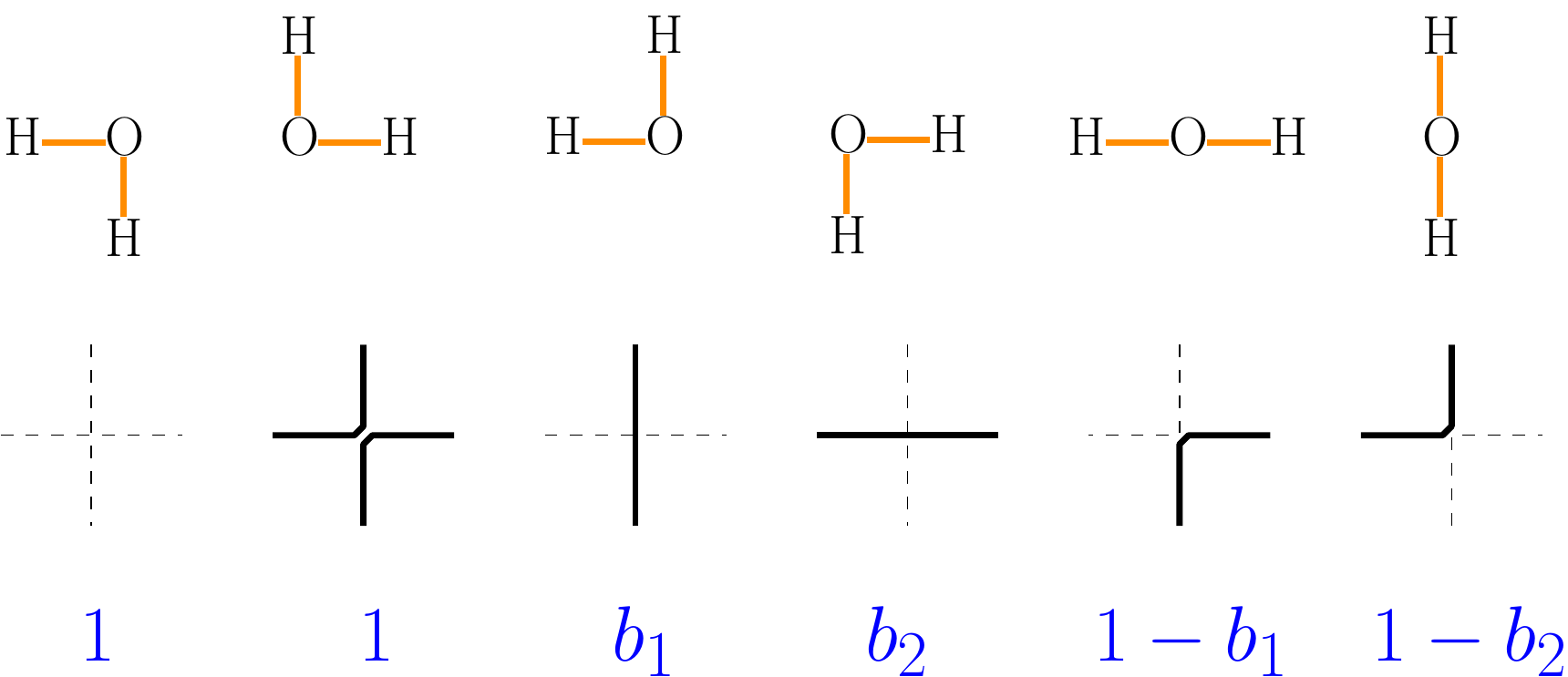}}}
 \end{center}
 \caption{Six types of vertices in $H_2O$ interpretation (top panel) and in lines interpretation (middle
 panel). Weights corresponding to the vertices for measure $\P(b_1,b_2)$ (bottom panel).
 \label{Figure_six}}
\end{figure}

\begin{figure}[h]
\begin{multicols}{2}
\begin{center}
  {\scalebox{0.43}{\includegraphics{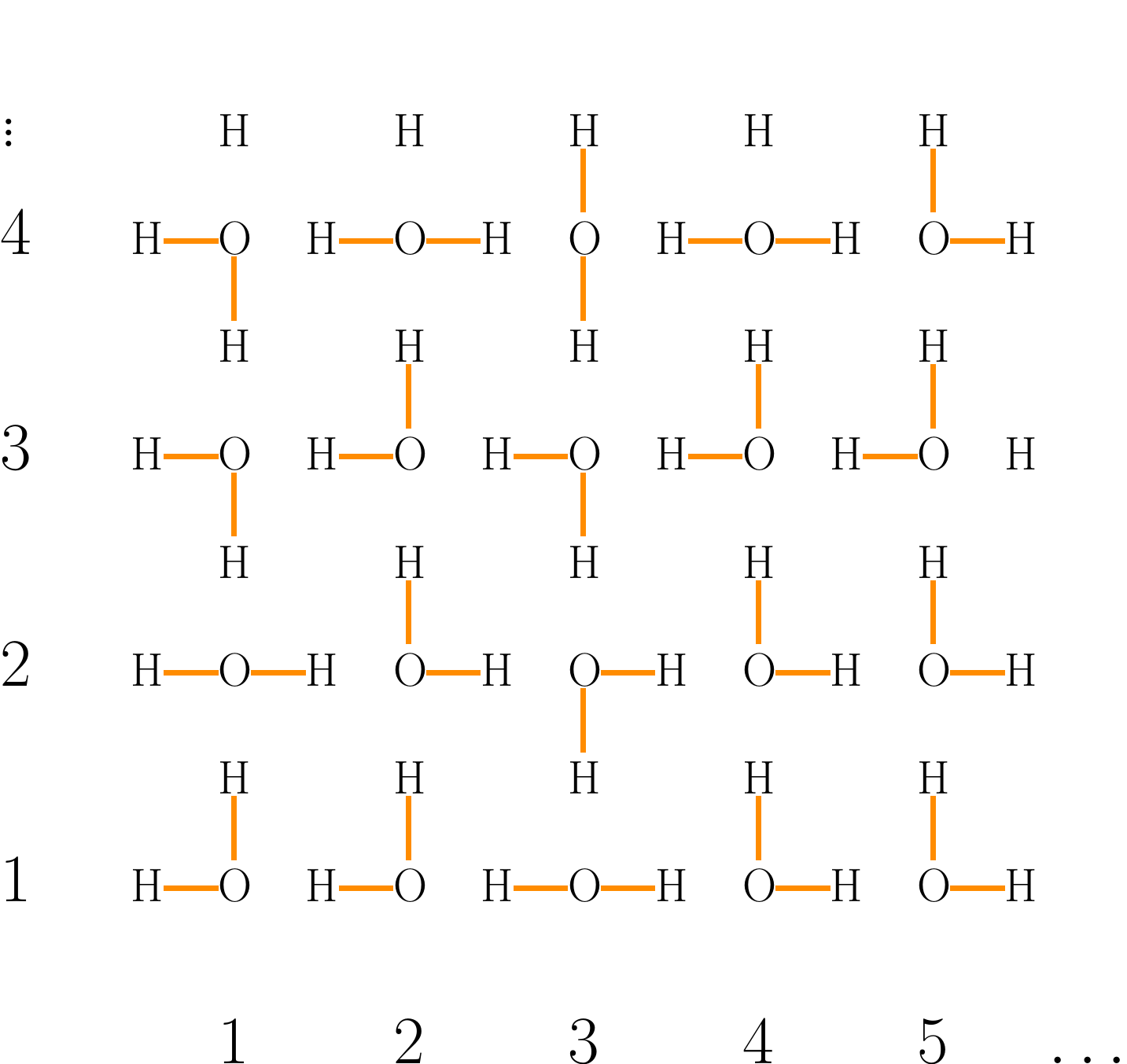}}}
 \end{center}
\columnbreak
\begin{center}
  {\scalebox{0.54}{\includegraphics{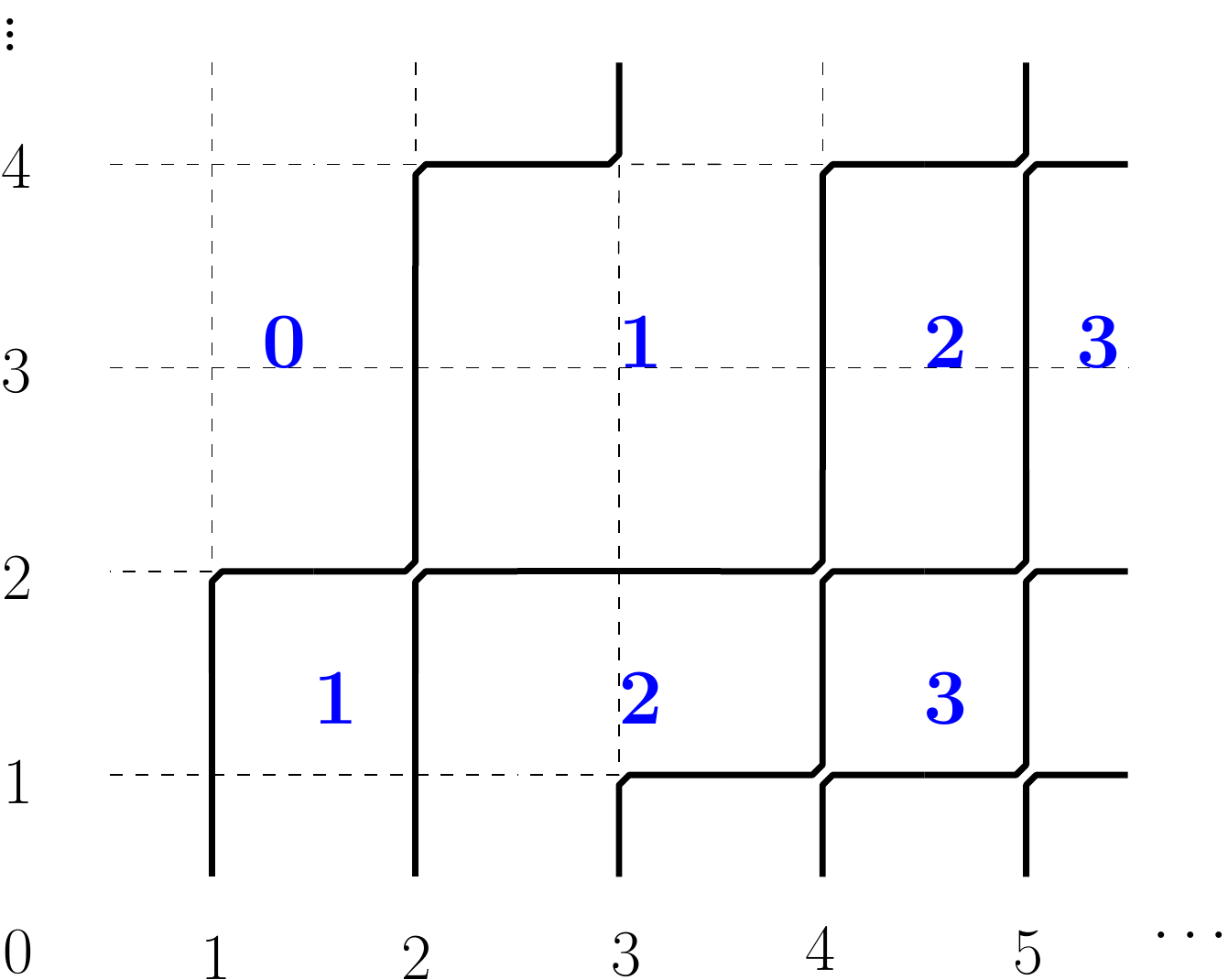}}}
 \end{center}

\end{multicols}

 \caption{Part of the configuration of the model in quadrant (left panel), corresponding line ensemble and height function (right panel).
 \label{Figure_configuration_H2O}}

\end{figure}

The six--vertex model is an important model of equilibrium statistical mechanics, being the
prototypical integrable lattice model in two dimensions. Its study has led to many exciting
developments during the last 50 years, see e.g.\ \cite{Baxter}, \cite{Resh_lectures} and
references therein. Our interest is probabilistic: we study random configurations and the
asymptotic distributions of the quantities describing them.

In the present paper we are concerned with the model in a quadrant, i.e.\ on the grid
$\mathbb{Z}_{>0}\times\mathbb{Z}_{>0}$, and with specific boundary conditions: $H$ atoms on the
left boundary and alternating $H$ and $O$ atoms on the bottom boundary, as shown in Figure
\ref{Figure_configuration_H2O}. Such boundary conditions can be viewed as an infinite analogue of
the well-studied domain--wall boundary conditions, cf.\ \cite{Bressoud}, \cite{Zinn_thesis},
\cite{Gier}, \cite[Introduction]{BFZ} for reviews of many recent results about the latter.

There is no canonical (e.g.\ uniform) measure on the configurations in the quadrant. Informally,
we are working with a special class of the \emph{asymmetric} six-vertex models whose parameters
fall on what is sometimes called the stochastic line, cf.\ \cite{GwaSpohn}, \cite{ADW},
\cite{Kim}, \cite{PS}. More formally, we introduce and study the family of measures $\P(b_1,b_2)$
depending on two parameters $0<b_1<1$ and $0<b_2<1$. The measure $\P(b_1,b_2)$ can be defined
through the following stochastic sampling algorithm. The types of vertices of
$\P(b_1,b_2)$--random configuration $\omega$ are chosen sequentially: we start from the corner
vertex at $(1,1)$, then proceed to $(1,2)$ and $(2,1)$,\dots, then proceed to all vertices $(x,y)$
with $x+y=k$, then with $x+y=k+1$, etc. The combinatorics of the model implies that when we choose
the type of the vertex $(x,y)$, then either it is uniquely determined by the types of its
previously chosen neighbors, or we need to choose between vertices of types number $3$ and number
$5$ at Figure \ref{Figure_six}, or we need to choose between vertices of types $4$ and $6$. We do
all the choices independently and choose type $3$ with probability $b_1$ and type $5$ with
probability $1-b_1$. Similarly we choose type $4$ with probability $b_2$ and type $6$ with
probability $1-b_2$. All these probabilities are recorded in Figure \ref{Figure_configuration_H2O}
and we refer to Section \ref{Section_P_as_Gibbs} for a more detailed description of $\P(b_1,b_2)$.

Alternatively, the measure $\P(b_1,b_2)$ can be obtained as a limit of \emph{Gibbs measures} in
certain finite domains approximating the quadrant, see Section \ref{Section_P_as_Gibbs} for the
exact statement. Furthermore, below and also in Section \ref{Section_P_as_interacting} we explain
that $\P(b_1,b_2)$--random configurations can be identified with time evolution of a certain
\emph{interacting particle system}.

In order to state our results we need to introduce the \emph{height function} of a model
configuration. For that we use a known identification of the six--vertex model with an ensemble of
paths, which is obtained by replacing the 6 types of $H_2O$ molecules with corresponding $6$ types
of local line configurations as shown in Figure \ref{Figure_six}. The result is an ensemble of
paths in the quadrant: all paths start at the bottom boundary and follow the up--right directions,
corners of the paths are allowed to touch, but paths never intersect each other, see Figure
\ref{Figure_configuration_H2O} for an example. These paths can be viewed as level lines of a
certain function. More precisely, for a configuration $\omega$ we define its height function
$H(x,y;\omega)$, $x,y\in \mathbb{R}_{>0}$, (here $x$ is the horizontal coordinate) by setting
$H(0,y;\omega)=0$ and declaring that when we cross a path, $H(x,y;\omega)$ increases by $1$, as
shown in Figure \ref{Figure_configuration_H2O}. A formal way of doing that is to define
$H(x,y;\omega)$ as the number of vertices of types $2$, $3$ and $5$ (non-strictly) to the left
from point $(x, \lceil y \rceil)$.

Our first result is the \emph{Law of Large Numbers} for the height function.

\begin{theorem} \label{Theorem_LLN} Assume that $0<b_2<b_1<1$ and let $\omega$ be a $\P(b_1,b_2)$--distributed random
configuration. Then for any $x,y>0$ the following convergence in probability holds
$$
 \lim_{L\to+\infty} \frac{H(Lx,Ly;\omega)}{L}=\H(x,y),
$$
where
$$
  \H (x,y)=\begin{dcases}
  \dfrac{\left(\sqrt{y(1-b_1)}-\sqrt{x(1-b_2)}
  \right)^2}{b_1-b_2},&\dfrac{1-b_1}{1-b_2}< \dfrac{x}{y} < \dfrac{1-b_2}{1-b_1},\\
  0,& \dfrac{x}{y}\le \dfrac{1-b_1}{1-b_2},\\
  x-y, & \dfrac{x}{y}\ge \dfrac{1-b_2}{1-b_1}.
  \end{dcases}
$$
\end{theorem}

We remark that for the six--vertex model in a finite (but growing) domain with fixed boundary
conditions, there is a general approach to the laws of large numbers similar to Theorem
\ref{Theorem_LLN} through the associated \emph{variational problem}, as is briefly outlined in
\cite{PR}. However, as far as the authors know, mathematically the variational approach has not
been yet developed to the point where it could produce a rigorous proof of Theorem
\ref{Theorem_LLN}; we use completely different methods to prove this theorem. Note that in a
related context of random height function arising from tilings and dimer models, the variational
approach is much more developed, see \cite{CKP}, \cite{KOS}, \cite{KO}. We will revisit the idea
of variational problems later from a point of view of interacting particle systems.

Our next result identifies the asymptotic fluctuations of the height function. Observe that the
limit shape $\H(x,y)$ is curved in the sector $\frac{1-b_1}{1-b_2}< \frac{x}{y}<
\frac{1-b_2}{1-b_2}$ and flat outside it (as $b_1\to b_2$ the curved region disappears).  The
regions where $\H(x,y)$ is flat are typically called \emph{frozen regions}, and one expects the
fluctuations to be exponentially small there. In the curved (also called ``liquid'') region the
situation is different.

\begin{theorem} \label{Theorem_fluctuations} Assume that $0<b_2<b_1<1$ and let $\omega$ be $\P(b_1,b_2)$--distributed random
configuration.  For any $x,y$ such that $\frac{1-b_1}{1-b_2}< \frac{x}{y}< \frac{1-b_2}{1-b_2}$,
and any  $s\in \R$ we have
$$
\lim_{L\to \infty} \PP\left(\frac{\H(x,y)L - H(Lx, Ly; \omega)}{\sigma_{x,y} L^{1/3}}\le s\right)
= F_{{\rm GUE}}(s),
$$
where
\begin{equation}
\label{eq_scaling_def_intro}
 \sigma_{x,y} =
\kappa^{-1/3}\frac{x^{1/6} y^{1/6} }{\kappa^{-1/2}-\kappa^{1/2}} \left(1-\sqrt{\kappa
x/y}\right)^{2/3}\left(1-\sqrt{ \kappa y/x}\right)^{2/3},\quad\quad \kappa=\frac{1-b_1}{1-b_2},
\end{equation}
  and
$F_{{\rm GUE}}$ is the GUE Tracy-Widom distribution.
\end{theorem}
\noindent{\bf Remark.}
 The symmetry $\sigma_{x,y}=\sigma_{y,x}$ can be traced to the fact that the measures
 $\P(b_1,b_2)$ are invariant under the involution which reflects the configuration
 by the $x=y$ line and then swaps the types of vertices in pairs $1\leftrightarrow 2$, $3\leftrightarrow
 4$, $5\leftrightarrow 6$. Another feature of \eqref{eq_scaling_def_intro} is that if we formally
 replace $\kappa$ by $\kappa^{-1}$ in the formula, then $\sigma_{x,y}$ changes its sign; we do not
 have any good explanations for the latter property.

\medskip

We recall that $F_{{\rm GUE}}$ is the limiting distribution for the largest eigenvalue of the
random Hermitian matrices (as the size of the matrix tends to infinity) from the \emph{Gaussian
Unitary Ensemble}, see \cite{TW-U}. One way to compute the distribution function $F_{{\rm
GUE}}(s)$ is through the Fredholm determinant expression:
$$
 F_{{\rm
GUE}}(s)=\det(I-A)_{L_2(s,+\infty)},
$$
where $A$ is an integral operator on $L_2(s,+\infty)$ with kernel expressed though the Airy
function via
$$
 A(x,y)=\frac{Ai(x)Ai'(y)-Ai(y)Ai'(x)}{x-y}.
$$

Theorem \ref{Theorem_fluctuations} is not the first instance of the appearance of a distribution
of the random matrix origin in the asymptotics of the six--vertex model. There is a class of
measures on configurations of the six--vertex model, called the \emph{free fermion point} of the
model, which can be analyzed via techniques of \emph{determinantal point processes}. In
particular, for the six--vertex model in $N\times N$ square with domain--wall boundary conditions
the study of these free fermion models is the same as the study of random \emph{domino tilings} of
the Aztec diamond, see \cite{EKLP}, \cite{Ku}, \cite{FS} for the details. Uniformly random tilings
of the Aztec diamond are known to posses random matrix asymptotic behavior, see \cite{J_Airy},
\cite{JN}. However, we do not know any direct relation of our measures $\P(b_1,b_2)$ with free
fermion point or determinantal point processes. Outside free fermions, the only connection to
random matrices that we are aware of, is the results of \cite{GP},\cite{G_ASM} for the six--vertex
model with domain--wall boundary conditions.

While the appearance of the random matrix type distribution in the six--vertex model is
anticipated, the exact form of Theorem \ref{Theorem_fluctuations} is a bit unexpected from the
point of view of statistical mechanics models with height functions. Indeed, in many related
models the asymptotic fluctuations of the height function are Gaussian and, more precisely, are
governed by the Gaussian Free Field, cf.\ \cite{Kenyon}, \cite{BF}, \cite{Petrov}, \cite{B-CLT},
\cite{BG_GFF}. Theorem \ref{Theorem_fluctuations} can probably be better understood from the point
of view of \emph{interacting particle systems} which we now present.

The connection of the six--vertex model to an interacting particle system with local interactions
was first noticed in \cite{GwaSpohn}. In order to see it in our context, we need to break the
symmetry between $x$ and $y$ coordinates. Consider $\P(b_1,b_2)$--distributed random configuration
and cut it by horizontal lines $y=t+1/2$, $t=0,1,2,\dots$ as shown in Figure
\ref{Figure_particles_intro}. The intersection of the horizontals with bold lines of the
configurations (that is, level lines of the height function) produce \emph{particle configuration}
$\x^{b_1,b_2}(t)=(x_1(t)<x_2(t)<\dots)$.

\begin{figure}[h]
\begin{center}
 {\scalebox{0.8}{\includegraphics{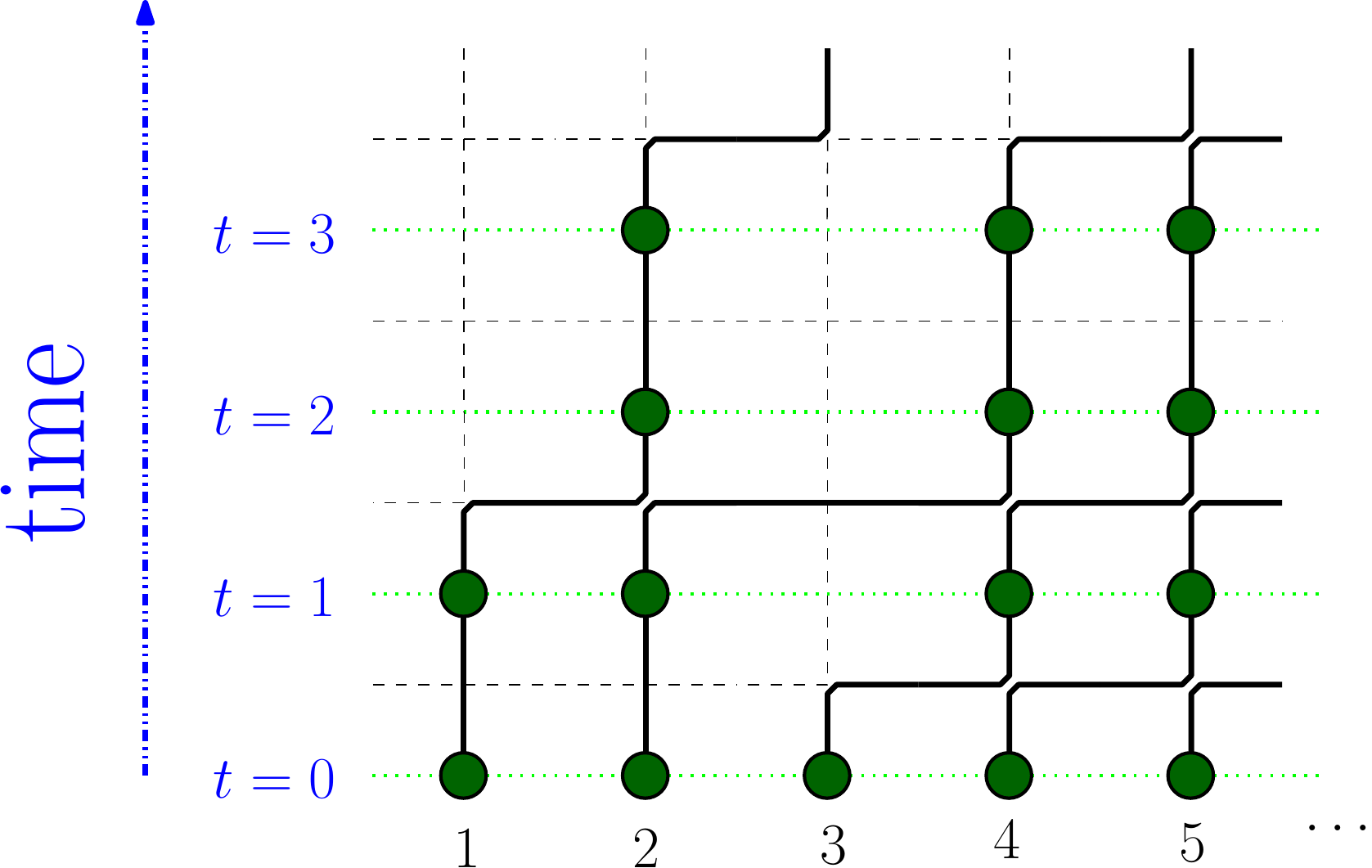}}}
 \caption{$5\times 4$ part of a configuration of the six--vertex model in the quadrant and
 corresponding particle system.
 \label{Figure_particles_intro}}
\end{center}
\end{figure}

The definition of the measure $\P(b_1,b_2)$ readily implies that $\x^{b_1,b_2}(t)$, $t=0,1,\dots$,
is a Markov chain. Moreover, this Markov chain has local update rules, as we explain in more
details in Section \ref{Section_P_as_interacting}.

From the point of view of $\x^{b_1,b_2}(t)$, Theorem \ref{Theorem_LLN} becomes the law of large
numbers for the large time \emph{current} of an interacting particle system. Thus, this result may
also be derivable through hydrodynamic theory as the solution to a one-dimensional variational
problem (cf.\ \cite{Spohn_book}, \cite{Rost}, \cite{Sep}, \cite{Rez}, \cite{Spohn_KPZ}) or as the
solution to a Hamilton-Jacobi PDE with flux which was computed in \cite{GwaSpohn}. The curved or
liquid part of the limit shape $\H(x,y)$ in the interacting particle systems language is known as
the \emph{rarefaction fan}. Further, the features of Theorem \ref{Theorem_fluctuations}, i.e.\ the
scaling exponent $L^{1/3}$ and the appearance of $F_{GUE}$ are a manifestation of the fact that
$\x^{b_1,b_2}(t)$ belongs to the \emph{KPZ universality class}, see \cite{Corwin} for a recent
review of KPZ universality. Note that the $L^{1/3}$ exponent was first predicted in this model by
\cite{GwaSpohn} based on the computation of the mixing time for the same interacting particle
system on a cylinder. Both $L^{1/3}$ exponent and the limiting distribution $F_{GUE}$ were found
before in several other interacting particle systems starting from the work \cite{J_TASEP} on the
Totally Asymmetric Simple Exclusion Process (TASEP). The latter (as well as other systems, see
\cite{FS2} for a survey) can be analyzed essentially using the techniques of determinantal point
processes. More recently and using a somewhat different set of techniques, similar results were
obtained for more complicated systems including Asymmetric Simple Exclusion Process (ASEP)
\cite{TW_review}, \cite{BCS}, $q$--TASEP \cite{BigMac}, \cite{F_qT}, and the solution to KPZ
stochastic partial differential equation \cite{ACQ}, \cite{SS}, \cite{Do}, \cite{CDR}, \cite{BCF};
see also \cite{BG_lect}, \cite{BP_lect}, \cite{Cor_ICM} for reviews.

The process $\x^{b_1,b_2}(t)$ has a limit to ASEP when $b_1\to 0$ with $b_2/b_1$ fixed, see
Section \ref{Section_P_as_interacting} for more details. Further limits include TASEP ($b_2=0$,
$b_1\to 0$) and KPZ stochastic partial differential equation (see \cite{KPZ}, \cite{BeGi},
\cite{ACQ}).

The relation to the KPZ universality class leads to predictions on \emph{multipoint} fluctuations
of the height function $H(x,y;\omega)$. For instance, we expect (but the proof is out of reach
with our present techniques) that if we fix $x,y$ satisfying $\frac{1-b_1}{1-b_2}< \frac{x}{y}<
\frac{1-b_2}{1-b_2}$, and vary $\zeta\in\mathbb R$, then the (proper centered and scaled) random
process $H(Lx+L^{2/3}\zeta, Ly;\omega)$ converges in distribution to the Airy$_2$ process (i.e.\
the top line of the Airy line ensemble). We refer to
\cite{PS_PNG},\cite{J_PNG},\cite{J_Airy},\cite{BF-Push} for similar statements for the interacting
particle systems related to the determinantal point processes.

\bigskip

The techniques we use to prove Theorems \ref{Theorem_LLN} and \ref{Theorem_fluctuations} also rely
on the interacting particle system interpretation. We start by finding the eigenfunctions of the
matrix of transitional probabilities of $\x^{b_1,b_2}(t)$ (when there are only finitely many
particles). Similar computations go back to the seminal work \cite{Lieb} on the six--vertex model
in 60s. The eigenrelation leads to contour integral formulas for the transitional probabilities of
$\x^{b_1,b_2}(t)$ which turns out to be very similar to those of the ASEP. This allows us to use
techniques developed in \cite{TW_determ} to find certain marginals of the distribution of
$\x^{b_1,b_2}(t)$. We further manipulate these expressions and use techniques developed in
\cite{BigMac}, \cite{BCS} (for the analysis of $q$--TASEP and certain directed polymers) to obtain
a Fredholm determinant formula for the $q$--Laplace transform of one-point distribution for the
height function $H(x,y;\omega)$. A steepest descent analysis of the kernel for this Fredholm
determinant ultimately leads to the results of Theorems \ref{Theorem_LLN} and
\ref{Theorem_fluctuations}.

\bigskip

\noindent{\bf Acknowledgements.} We would like to thank H.~Spohn for useful discussions. A.B.\ was
partially supported by the NSF grant DMS-1056390. I.C.\ was partially supported by the NSF through
DMS-1208998 as well as by Microsoft Research and MIT through the Schramm Memorial Fellowship, by
the Clay Mathematics Institute through the Clay Research Fellowship, and by the Institute Henri
Poincare through the Poincare Chair. V.~G.\ was partially supported by the NSF grant DMS-1407562.

\section{Two formulations of the model}

\label{Section_P}

Let us start by giving a formal definition of our main subject of study --- the measure
$\P(b_1,b_2)$. For convenience we stick to the bold lines interpretation of the six--vertex model,
the $H_2O$ molecules picture can be always restored through the bijection of Figure
\ref{Figure_six}.

The configuration of the model is an assignment of the pictures of 6 kinds shown in Figure
\ref{Figure_weights} to the vertices of the grid $\Z_{>0}\times \Z_{>0}$ subject to three
restrictions:
\begin{figure}[h]
\begin{center}
 {\scalebox{0.8}{\includegraphics{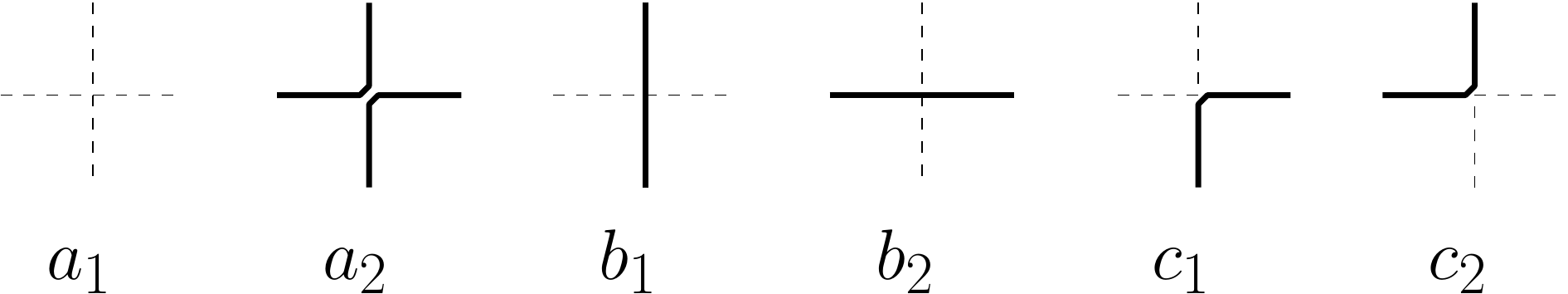}}}
 \caption{Six types of vertices, corresponding lines and weights.
 \label{Figure_weights}}
\end{center}
\end{figure}
\begin{itemize}
\item The pictures at any two adjacent sites agree in the sense that the bold lines keep flowing.
E.g. if a picture at position $(x,y)$ has a bold line in west direction, then the picture at
position $(x-1,y)$ should have a bold line in east direction.
\item Pictures at positions $(x,1)$ for $x>0$ have bold lines in south direction.
\item Pictures at positions $(1,y)$ for $y>0$ do not have bold lines in west direction.
\end{itemize}


Let $\Omega$ denote the space of all configurations satisfying the above conditions. One element
from $\Omega$ is shown in the right panel of Figure \ref{Figure_configuration_H2O}. We remark that
traditionally, one subdivides the six--types into $3$ pairs $a_1$, $a_2$, $b_1$, $b_2$, $c_1$,
$c_2$, as shown in Figure \ref{Figure_weights}, and we use such notation for the types of
vertices.

We now explain a procedure for sampling a $\P(b_1,b_2)$--distributed element $\omega\in \Omega$,
thus defining the measure.

Let $\omega(x,y)$ denote the type of the vertex at $(x,y)$ and proceed inductively in $x+y$, i.e.\
we first sample the vertices with $x+y=2$, then with $x+y=3$, etc. When we sample $\omega(x,y)$,
the vertices $\omega(x-1,y)$ and $\omega(x,y-1)$ are already known, therefore, we know whether the
south and west lines in $\omega(x,y)$ should be bold or not. Thus, we have four cases:
\begin{itemize}
 \item If both west and south lines in $\omega(x,y)$ are bold, then we set $\omega(x,y)$ to
 be of type $a_2$.
 \item If both west and south lines in $\omega(x,y)$ are not bold, then we set $\omega(x,y)$ to
 be of type $a_1$.
 \item If the west line in $\omega(x,y)$ is bold, while the south line in $\omega(x,y)$
 is not bold, then we set $\omega(x,y)$ to be of type $b_2$ with probability $b_2$
 and to be of type $c_2$ with probability $1-b_2$.
 \item If the west line in $\omega(x,y)$ is not bold, while the south line in $\omega(x,y)$
 is bold, then we set $\omega(x,y)$ to be of type $b_1$ with probability $b_1$
 and to be of type $c_1$ with probability $1-b_1$.
\end{itemize}

\begin{definition} $\P(b_1,b_2)$ is the probability measure on $\Omega$ obtained by the
above sampling procedure. \label{Definition_P}
\end{definition}

The measures $\P(b_1,b_2)$ can be put into two different contexts. In Section
\ref{Section_P_as_Gibbs} we explain that these measures are specific Gibbs measures for the
six--vertex model with a quadratic relation on weights of the vertices $a_1,a_2,b_1,b_2,c_1,c_2$.
In Section \ref{Section_P_as_interacting} we explain that $\P(b_1,b_2)$ can be viewed as the law
of an interacting particle system.

\subsection{$\P(b_1,b_2)$ as a Gibbs measure for the six--vertex model}

\label{Section_P_as_Gibbs}

We want to study probability measures on $\Omega$ satisfying the Gibbs property encoded by $6$
positive weights $a_1, a_2$, $b_1, b_2$, $c_1, c_2$. The identification of weights and types of
vertices is shown in Figure \ref{Figure_weights} and somewhat abusing the notations we will use
the same notation for them; e.g.\ we write $a_1$ both for the type of a vertex and for the
corresponding positive weight.

\begin{definition}
 A probability measure $P$ on $\Omega$ is called Gibbs, if for any finite subdomain $S\subset \Z_{>0}\times
\Z_{>0} $, the conditional distribution of $P$ on the configurations inside $S$ given the
 configuration outside $S$ has weights proportional to
 $$
  (a_1)^{\#(a_1)}   (a_2)^{\#(a_2)}   (b_1)^{\#(b_1)}   (b_2)^{\#(b_2)}   (c_1)^{\#(c_1)}
  (c_2)^{\#(c_2)},
 $$
 where $\#(a_1)$ means the number of vertices of type $a_1$ inside $S$ and similarly for other
 types.
\end{definition}

We are not aware of any classification theorems for the Gibbs measures on $\Omega$. However, we
note that in related contexts theorems of this kind are known (cf.\ \cite{Shef}, \cite{GH},
\cite{G-A}).

We will obtain our Gibbs measures as limits of the Gibbs measures on growing finite subdomains of
$\Omega$.

For two positive integers $X$ and $Y$ let $\Omega(X,Y)$ denote the set of the configurations of
the six--vertex model in $\{1,\dots,X\}\times \{1,\dots,Y\}$ satisfying the same boundary
condition as configurations from $\Omega$ along the south and west boundaries and no boundary
conditions along the north and east boundaries. For instance, the right panel of Figure
\ref{Figure_configuration_H2O} can be viewed as a configuration from $\Omega(5,4)$. There are
finitely many configurations in $\Omega(X,Y)$ and we equip $\Omega(X,Y)$ with a Gibbs probability
measure $P(X,Y; a_1,a_2,b_1,b_2,c_1,c_2)$ such that the probability of $\omega\in\Omega(X,Y)$
having $\#(a_1)$ vertices of type $a_1$, $\#(a_2)$ vertices of type $a_2$, etc., is
$$
 \frac{ (a_1)^{\#(a_1)}   (a_2)^{\#(a_2)}   (b_1)^{\#(b_1)}   (b_2)^{\#(b_2)}   (c_1)^{\#(c_1)}
  (c_2)^{\#(c_2)}}{Z(X,Y; a_1,a_2,b_1,b_2,c_1,c_2)},
$$
where $Z(\cdot)$ is a normalization constant.

Clearly, $\Omega$ can be viewed as an a limit of $\Omega(X,Y)$ as $X,Y\to\infty$. We say that a
probability measure $M$ is a limit of measures $M_n$ on $\Omega(X_n,Y_n)$ as $n\to\infty$ if
$\lim_{n\to\infty} X_n=\lim_{n\to\infty} Y_n=\infty$ and for any finite collection $(x_1,y_1)$,
\dots $(x_k,y_k)$ and for any set of types $t_i\in\{a_1,a_2,b_1,b_2,c_1,c_2\}$, $i=1,\dots,k$, the
$M$--probability of having the vertices of types $t_1$, \dots, $t_k$ in positions $(x_1,y_1)$,
\dots, $(x_k,y_k)$ is the $n\to\infty$ limit of the $M_n$--probabilities for the same event. We
similarly define taking limits as only one of $X_n$, $Y_n$ tends to infinity, or if one of $X$ and
$Y$ is already infinite.

\begin{proposition} \mbox{}
\label{Proposition_measure} \begin{enumerate}\item Suppose that the weights of the six--vertex
model satisfy the quadratic relation
\begin{equation}
\label{eq_quadratic_condition}
 c_1 c_2= (a_1-b_2) (a_2-b_1),
\end{equation}
and inequality $a_1>b_2$, which guarantees that the factors in \eqref{eq_quadratic_condition} are
positive. Then the measure $\P$ (from Definition \ref{Definition_P})  can be obtained through the
following iterated limit:
\begin{equation}
\label{eq_iterated_limit}
 \P\left(\frac{b_1}{a_2},\frac{b_2}{a_1}\right)=  \lim_{Y\to\infty} \lim_{X\to\infty} P(X,Y;
 a_1,a_2,b_1,b_2,c_1,c_2),
\end{equation}
where the limits are understood as explained above.
\item Suppose that the weights of the six--vertex model satisfy the quadratic relation
\begin{equation}
\label{eq_quadratic_condition_2}
 c_1 c_2= (a_1-b_1) (a_2-b_2),
\end{equation}
and inequality $a_1>b_1$, which guarantees that the factors in \eqref{eq_quadratic_condition_2}
are positive. Then the measures $\P$ (from Definition \ref{Definition_P}) can be obtained through
the following iterated limit:
\begin{equation}
\label{eq_iterated_limit_2}
 \P\left(\frac{b_1}{a_1},\frac{b_2}{a_2}\right)=  \lim_{X\to\infty} \lim_{Y\to\infty} P(X,Y; a_1,a_2,b_1,b_2,c_1,c_2).
\end{equation}
\end{enumerate}
\end{proposition}
\noindent{\bf Remark 1.} It is very plausible that the limits in \eqref{eq_iterated_limit},
\eqref{eq_iterated_limit_2} exist also when the conditions \eqref{eq_quadratic_condition},
\eqref{eq_quadratic_condition_2} are not satisfied, and, even further, when instead of iterative
limit, we send $X$ and $Y$ to $\infty$ simultaneously in certain regular ways. However, the
limiting measures would probably be different from the family $\P( b_1,b_2)$ and we do not analyze
such limits of measures in the present article.

\noindent{\bf Remark 2.} The quadratic relation \eqref{eq_quadratic_condition} was introduced in
 \cite{GwaSpohn}. Under this condition the transfer matrix (see Section \ref{Section_transfer})
 can be easily renormalized to become
 stochastic (and related to interacting particle systems described in Section \ref{Section_P_as_interacting}).

\begin{proof}[Proof of Proposition \ref{Proposition_measure}] We first analyze $\lim_{Y\to\infty} \lim_{X\to\infty}
P(X,Y; a_1,a_2,b_1,b_2,c_1,c_2)$ under condition \eqref{eq_quadratic_condition}.

Take a fixed finite $Y$ and large $X$ and consider the $P(X,Y;
a_1,a_2,b_1,b_2,c_1,c_2)$--distributed configuration $\omega$. Note that our boundary condition
implies that in the horizontal row $y$ there are at most $y$ vertices of types $a_1$, $b_2$,
$c_1$, $c_2$ and all others are either of type $a_2$ or $b_1$. Furthermore, the vertices of types
$a_2$ and $b_1$ clearly cannot be horizontally adjacent. Together with the inequality $a_2>b_1$
this implies that with probability tending to $1$ as $X\to\infty$ all the vertices in the vertical
column $X$ are of type $a_2$. As a conclusion, tracing the paths which start at the bottom
boundary, we see that with probability tending to $1$ the total number of the vertices of types
$a_2$, $b_1$ and $c_1$ is equal to $XY-(1+2+\dots +Y)$. In particular, this number becomes
deterministic. Since the number of all vertices is always $XY$, the total number of the vertices
of types $a_1$, $b_2$ and $c_2$ also becomes deterministic.

In addition, looking at horizontal bold lines, we see that in each horizontal row $y$, the number
of the vertices of type $c_1$ is $1$ plus the number of the vertices of type $c_2$. Therefore, the
difference between the total number of the vertices of types $c_1$ and $c_2$ also becomes
deterministic.

The argument of the previous two paragraph implies now that if we multiply the weights $a_2$,
$b_1$ and $c_1$ by a constant $U$, the weights $a_1$, $b_2$ and $c_2$ by another constant $V$, the
weight $c_1$ by another constant $W$ and the weight $c_2$ by $1/W$, then the limit $
\lim_{X\to\infty} P(X,Y; a_1,a_2,b_1,b_2,c_1,c_2)$ will not change. Choosing appropriate constants
$U=\frac{1}{a_2}$, $V=\frac{1}{a_1}$, $W=\frac{1-b_1/a_2}{c_1/a_2}$, taking into the account
\eqref{eq_quadratic_condition} and adapting the notations $\tilde b_1=b_1/a_2$, $\tilde
b_2=b_2/a_1$, we can transform the six weights $(a_1,a_2,b_1,b_2,c_1,c_2)$ into new six weights
$(1,1,\tilde b_1,\tilde b_2, 1-\tilde b_1, 1-\tilde b_2)$. Under new weights we immediately see
the convergence to the above description for sampling $\P(\tilde b_1,\tilde b_2)$ (in fact, for
the new parameters the same algorithm works for any finite domain $\Omega(X,Y)$).

\smallskip
In order to analyze the limit in the different order $\lim_{X\to\infty} \lim_{Y\to\infty} P(X,Y;
a_1,a_2,b_1,b_2,c_1,c_2)$, note that $\Omega$ possesses the following involutive symmetry: take a
configuration $\omega\in\Omega$, reflect it by the main diagonal $x=y$ and further interchange the
bold and dotted lines of Figure \ref{Figure_weights}. As a result we obtain another configuration
from $\Omega$. Note that under this transformation, the vertices of types $b_1$, $b_2$, $c_1$,
$c_2$ turn into the vertices of the same type (in reflected positions), while the vertices of
types $a_1$ and $a_2$ swap their types. Hence,  \eqref{eq_quadratic_condition_2} and
\eqref{eq_iterated_limit_2} are obtained from \eqref{eq_quadratic_condition} and
\eqref{eq_iterated_limit} by this involution.
\end{proof}

\subsection{$\P(b_1,b_2)$ as an interacting particle system}
\label{Section_P_as_interacting}

It was observed in \cite{GwaSpohn} that under the quadratic condition
\eqref{eq_quadratic_condition} the six--vertex model can be naturally related to a certain
interacting particle system. Let us describe how this works for our measures $\P(b_1,b_2)$.

\bigskip
Consider $\P(b_1,b_2)$--distributed random configuration and cut it by horizontal lines $y=t+1/2$,
$t=0,1,2,\dots$ as shown in Figure \ref{Figure_particles_intro}. The intersection of the
horizontals with bold lines  produces \emph{particle configuration}
$\x^{b_1,b_2}(t)=(x_1(t)<x_2(t)<\dots)$.

The definition of the measure $\P(b_1,b_2)$ via a local sampling procedure readily implies that
$\x^{b_1,b_2}(t)$, $t=0,1,\dots$ is a Markov chain and moreover can be viewed as an
\emph{interacting particle system} with local interactions. Indeed, given $\x^{b_1,b_2}(t)$, the
configuration $\x^{b_1,b_2}(t+1)$ can be defined as follows: we sequentially define the positions
$x_i(t+1)$ of the particles for $i=1,2,\dots$. The $i$th particle at position $x_i(t)$ jumps at
time $t+1$ to any position from the interval $\{\max(x_{i-1}(t+1)+1,x_i(t)),
\max(x_{i-1}(t+1)+1,x_i(t))+1,\dots,x_{i+1}(t) \}$ with the following probabilities:
\begin{itemize}
\item If $x_{i-1}(t+1)\le x_i(t)$, then
\begin{multline*}
 \PP \biggl(x_i(t+1)=x_i(t)+k \mid \x^{b_1,b_2}(t), x_{i-1}(t+1)\biggr)\\=\begin{cases} b_1,& k=0,\\
   (1-b_1)(1-b_2) b_2^{k-1}, & 0<k< x_{i+1}(t)-x_{i}(t),\\
   (1-b_1) b_2^{x_{i+1}(t)-x_{i}(t)-1}, &k= x_{i+1}(t)-x_{i}(t),\\
   0, & \text{otherwise.}
   \end{cases}
\end{multline*}
\item If $x_{i-1}(t+1)=x_i(t)$, then
$$
 \PP \biggl(x_i(t+1)=x_i(t)+k \mid \x(t), x_{i-1}(t+1)\biggr)=\begin{cases}
   (1-b_2) b_2^{k-1}, & 0<k< x_{i+1}(t)-x_{i}(t),\\
    b_2^{x_{i+1}(t)-x_{i}(t)-1}, &k= x_{i+1}(t)-x_{i}(t),\\
   0, & \text{otherwise.}
   \end{cases}
$$
\end{itemize}
Informally, one says that the $i$th particle flips a $b_1$--biased coin to decide whether it stays
on its position, or not. If not, then it jumps to the right with $b_2$--geometric probabilities;
when it reaches the $i+1$st particle, it can push it to the right, but at most by $1$. After that
the $i+1$st particle starts to move, etc.

The definition implies that $\x^{b_1,b_2}(0)=(1,2,3,\dots)$, such configuration at time $0$ is
usually referred to as \emph{step} initial condition. Note that the state space of the Markov
chain $\x^{b_1,b_2}(t)$ is \emph{countable}, which is a consequence of the choice of the initial
condition, see Section \ref{Section_infinitely_many} for more details.

\medskip

There is a connection of our interacting particle system to the Asymmetric Simple Exclusion
Process (ASEP). Recall that ASEP is an interacting particle system on $\mathbb Z$ in continuous
time in which each particle has two (exponential) clocks: the right clock of intensity $p$ and the
left clock of intensity $q$. When the right clock rings the particle checks whether the position
to its right is free: if Yes, then it jump by $1$ to the right, if No, then nothing happens. When
the left clock rings the particle checks whether the position to its left is free: if Yes, then it
jump by $1$ to the left, if No, then nothing happens. Afterwards the clocks are restarted. The
ASEP is a very well studied model, cf.\ \cite{Spitzer}, \cite{Li1}, \cite{Li2}, \cite{TW_review},
\cite{C_ASEP}.

We are not going to provide the proof, but it is very plausible that the following is true: for
any two reals $p,q>0$
\begin{equation}
\label{eq_limit_to_ASEP}
  \lim_{\eps\to 0} \x^{\eps q,\eps p}(\lfloor \eps^{-1}t\rfloor)-\eps^{-1} t = \y^{p,q}(t),
\end{equation}
 where $\y^{p,q}(t)$ is ASEP started with step initial condition, i.e.\ $\y^{p,q}(0)=(1,2,3,\dots)$.
It is important to emphasize the subtraction of $\eps^{-1} t$ in \eqref{eq_limit_to_ASEP}. In
other words, we observe ASEP near the diagonal. Note that for $k$--particle version of
$\x^{b_1,b_2}(t)$ an analogue of \eqref{eq_limit_to_ASEP} is straightforward, moreover, it will be
clear below (see Theorem \ref{Theorem_eigentransfer} and remark after it) that the formulas we get
for the transitional probabilities of the $k$--particle version of $\x^{b_1,b_2}(t)$ converge to
those for $k$--particle ASEP. However, in order to rigorously prove \eqref{eq_limit_to_ASEP} one
would need to deal with infinitely many particles which we leave out of the scope of the present
article.

\bigskip

Another curious limit is obtained if we send $b_1\to 1$ (we again do not provide a complete
proof). Set
\begin{equation}
\label{eq_limit_to_waking}
 \z^{b}(t)= \lim_{\eps\to 0} \x^{\eps^{-1},b}(\lfloor \eps^{-1}t\rfloor).
\end{equation}
The dynamics $\z^b(t)$ has the following description: each particle has an independent exponential
clock of rate $1$. When the clock of $i$th particle rings at time $t$, it wakes up and jumps to
the right by $k$ steps  $(0<k< x_{i+1}(t)-x_{i}(t))$ with probability $(1-b)b^{k-1}$. With
remaining probability  $b^{x_{i+1}(t)-x_{i}(t)-1}$ the $i$th particle jumps to the position of
$x_{i+1}(t)$ of the $(i+1)$st particle and in this latter case the $(i+1)$st particle also wakes
up and repeats the same procedure, etc.

\section{Transfer matrices}
\label{Section_transfer}

Proofs of many results in statistical mechanics use the transfer matrices, including the
celebrated computation of the free energy in the six vertex model, cf.\ \cite{Lieb}, \cite{Baxter}
and references therein. They are crucial for our analysis as well.

\subsection{Finitely many lines}

Fix an integer $N>0$ and let $\W_N$ denote the set of ordered $N$-tuples of integers
$x_1<x_2<\dots<x_N$.

Let $\mathcal L$ denote one-row configurations of the six--vertex model, i.e.\ these are the
assignments of six types of vertices of Figure \ref{Figure_weights} to $\mathbb Z\times \{1\}$
such that the bold line configurations for all the adjacent vertices agree with each other. Take
two elements $\x,\y\in\W_N$. We say that $\ell\in \mathcal L$ sends $\x$ to $\y$ if all the bold
south lines in $\ell$ are at the positions of $\x$ and all the bold north lines are at the
positions of $\y$, cf.\ Figure \ref{Figure_one_line}. Note that for each pair $\x,\y\in\W_N$ there
is at most one $\ell\in \mathcal L$ sending $\x$ to $\y$.

\begin{figure}[h]
\begin{center}
 {\scalebox{0.75}{\includegraphics{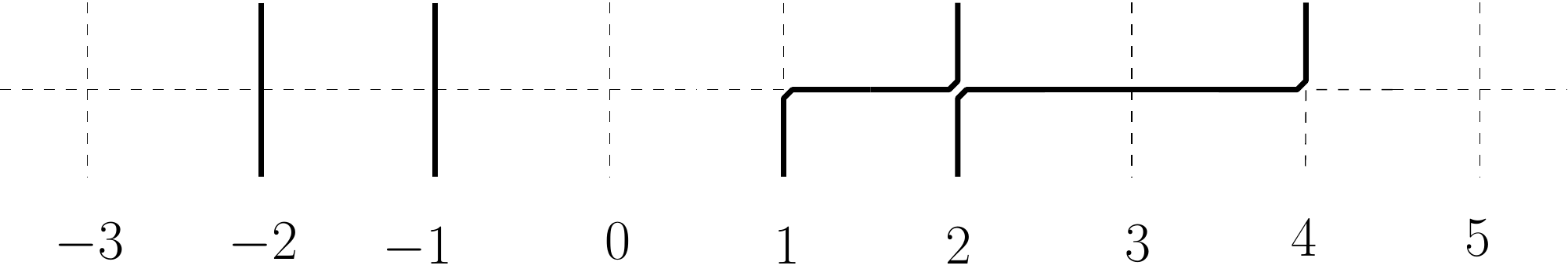}}}
 \caption{An element from $\mathcal L$ sending $(-2<-1<1<2)\in \W_4$ to $(-2<-1<2<4)\in \W_4$.
 \label{Figure_one_line}}
\end{center}
\end{figure}

Now fix six weight parameters $(a_1,a_2,b_1,b_2,c_1,c_2)$ such that $a_1=1$ and define \emph{the
transfer matrix} $\T^{(N)}(\x\to\y; 1,a_2,b_1,b_2,c_1,c_2)$ with rows and columns parameterized by
$\W_N$ via
\begin{multline}\label{eq_def_T}
 \T^{(N)}(\x\to\y; 1,a_2,b_1,b_2,c_1,c_2)\\=\begin{cases}   (a_2)^{\#(a_2[\ell])}   (b_1)^{\#(b_1[\ell])}   (b_2)^{\#(b_2[\ell])}   (c_1)^{\#(c_1[\ell])}
  (c_2)^{\#(c_2[\ell])},& \ell \text{ sends } \x\text{ to }\y,\\
  0,& \text{there is no such }\ell,
  \end{cases}
\end{multline}
where $\#(a_2[\ell])$ is the number of vertices of type $a_2$ in $\ell$, and similarly for other
types of vertices.

The quadratic relation \eqref{eq_quadratic_condition} can be interpreted in terms of matrices
$\T^{(N)}$, as was first observed in \cite{GwaSpohn}.

\begin{proposition} \label{Proposition_normalization}
 The matrices $\T^{(N)}(\x\to\y; 1,a_2,b_1,b_2,c_1,c_2)$ are stochastic for all $N=1,2,\dots$ if and only if
 \begin{equation}
 \label{eq_stochastic_condition}
  a_2=1,\quad c_1 c_2=(1-b_1)(1-b_2),\quad   b_1<1.
 \end{equation}
 The matrices $\T^{(N)}(\x\to\y; 1,a_2,b_1,b_2,c_1,c_2)$ can be normalized to be stochastic, i.e.\ for every $N\in\mathbb Z_{>0}$ and
 $\x\in \W_N$ we have
 $$
  \sum_{Y\in \W_N}  \T^{(N)}(\x\to\y; 1,a_2,b_1,b_2,c_1,c_2) = C_N,\quad \text{ where }C_N \text{ does not depend on }\x,
 $$
 if and only if
 \begin{equation}
\label{eq_GS_condition} c_1 c_2= (a_1-b_2)(a_2-b_1), \quad a_1=1,\quad b_2<1.
\end{equation}
 In the latter case the \emph{normalized} transfer matrix can be identified with another transfer
 matrix with six weights $(1,1,b_1/a_2, b_2/a_1, 1-b_1/a_2, 1-b_2/a_1)$.
\end{proposition}
\noindent{\bf Remark.}
 When \eqref{eq_stochastic_condition} is satisfied, the matrices $\T^{(N)}(\x\to\y;
 1,a_2,b_1,b_2,c_1,c_2)$ are precisely transition probabilities of  the finite number of particles
 version of the Markov chain $\x^{b_1,b_2}(t)$ discussed in Section
 \ref{Section_P_as_interacting}.
\begin{proof}[Proof of Proposition \ref{Proposition_normalization}]
We start with one particle, i.e.\ $N=1$ and $\x=x_1$. There are two possibilities for $\y=y_1$
with non-vanishing $\T^{(N)}(\x\to\y;1,a_2,b_1,b_2,c_1,c_2)$:
\begin{itemize}
 \item $y_1=x_1$, then $\T^{(N)}(\x\to\y;1,a_2,b_1,b_2,c_1,c_2)=c_1c_2$,
 \item $y_1=x_1+k$, $k=1,2,\dots$, then $\T^{(N)}(\x\to\y;1,a_2,b_1,b_2,c_1,c_2)=b_2^{k-1} c_1 c_2.$
\end{itemize}
 The total sum of weights is (assuming $b_2<1$, otherwise it diverges)
\begin{equation}
\label{eq_1_part_sum}
 \sum_{\y\in \W_1} \T^{(N)}(\x\to\y;1,a_2,b_1,b_2,c_1,c_2)=b_1+\frac{c_1 c_2}{1-b_2}.
\end{equation}
 This sum is $1$ and the transition
matrix is stochastic if and only if $c_1 c_2=(1-b_1)(1-b_2)$. Note also that the sum
\eqref{eq_1_part_sum} does not depend on $\x$. Thus, even if \eqref{eq_1_part_sum} is not $1$, we
still can renormalize the matrix $\T^{(1)}$ to make it stochastic.

We continue with two particle configurations ($N=2$) with  $\x=(x_1<x_2)$. There are several
possibilities for $\y=(y_1<y_2)$ with nonzero $\T^{(2)}(\x\to\y;1,a_2,b_1,b_2,c_1,c_2)$:
\begin{itemize}
 \item $y_1=x_1$, $y_2=x_2$, then $\T^{(N)}(\x\to\y)=b_1 b_1$.
 \item $y_1=x_1$, $y_2=x_2+k$, $k=1,2,\dots$, then $\T^{(N)}(\x\to\y)=b_1 c_1 c_2 b_2^{k-1}$,
 \item $y_1=x_1+k$, $k=1,\dots,x_2-x_1-1$, $y_2=x_2$, then $\T^{(N)}(\x\to\y)=b_1 c_1 c_2 b_2^{k-1}$
 \item $y_1=x_1+k$, $k=1,\dots,x_2-x_1-1$, $y_2=x_2+\ell$, $\ell=1,2,\dots$, then
$\T^{(N)}(\x\to\y)= c_1 c_2 b_2^{k-1} c_1 c_2 b_2^{\ell-1}$.
 \item $y_1=x_2$, $y_2=x_2+\ell$, $\ell=1,2,\dots$, then $\T^{(N)}(\x\to\y)= c_1 a_2 b_2^{x_2-x_1-1}
c_2 b_2^{\ell-1}$.
\end{itemize}
Let us do the summation over $\y\in\W_2$. The first four scenarios give (again, assuming $b_2<1$):
$$
\left(b_1 + c_1c_2 \frac{1-b_2^{x_2-x_1-1}}{1-b_2}\right)\cdot
\left(b_1+\frac{c_1c_2}{1-b_2}\right).
$$
The fifth scenario gives
$$
 c_1c_2 a_2 \frac{b_2^{x_2-x_1-1}}{1-b_2}.
$$
The total sum
\begin{equation}
\label{eq_norm_2} \sum_{\y\in \W_2}\T^{(2)}(\x\to\y;1,a_2,b_1,b_2,c_1,c_2)=
\left(b_1+\frac{c_1c_2}{1-b_2}\right)^2 + c_1c_2
\frac{b_2^{x_2-x_1-1}}{1-b_2}\left(b_1+\frac{c_1c_2}{1-b_2}-a_2\right)
\end{equation}
is independent of $x_1$, $x_2$ if and only if
\begin{equation}
\label{eq_normalization_condition}
 b_1+\frac{c_1 c_2}{1-b_2}-a_2=0
\end{equation}
which is precisely the condition \eqref{eq_GS_condition}. Further, if
\eqref{eq_normalization_condition} holds, then \eqref{eq_norm_2} turns into $a_2^2$, so
stochasticity of $\T^{(2)}$ implies $a_2=1$. This argument proves that the restrictions
\eqref{eq_stochastic_condition} and \eqref{eq_GS_condition} and necessary.

Further, suppose that \eqref{eq_GS_condition} is satisfied. Note that if multiply the weights
$a_2$, $b_1$ and $c_1$ (these vertices have bold south edges) simultaneously by the same constant
$D$, then the weight of every configuration (i.e.\ every matrix element in $\T^{(N)}$) is
multiplied by $D^N$. With this transformation we can now set $a_2=1$. Then \eqref{eq_GS_condition}
turns into \eqref{eq_stochastic_condition} and this proves that thus normalized transfer matrix
fits into the form \eqref{eq_stochastic_condition}.

It remains to prove the sufficiency of \eqref{eq_stochastic_condition} for the stochasticity of
the transfer matrix $\T^{(N)}$ for all $N=1,2,3,\dots$.

Note the transfer-matrix $\T^{(N)}(\,\cdot\,;1,a_2,b_1,b_2,c_1,c_2)$ depends on $c_1$, $c_2$ only
through their product $c_1c_2$
--- indeed, the vertices of these two types
appear in pairs. Thus, under \eqref{eq_stochastic_condition} we can replace the weights by those
satisfying
\begin{equation}
\label{eq_stoch_condition_2}
 a_2=b_1+c_1=b_2+c_2=1,
\end{equation}
without changing the transfer matrix. Now the application of such transfer matrix has a clear
stochastic meaning along the lines of the definition of measures $\P(b_1,b_2)$ in Section
\ref{Section_P} and the interacting particle system in Section \ref{Section_P_as_interacting}.
Namely we define the types of vertices sequentially from left to the right. Then for each vertex
either we insert a type $a_1/a_2$ vertex (whose type is already uniquely defined), or we choose
between $b_1$ and $c_1$ with corresponding probabilities, or we choose between $b_2$ and $c_2$
with corresponding probabilities. This proves that under the condition
\eqref{eq_stoch_condition_2} (or, equivalently, \eqref{eq_stochastic_condition}) the transfer
matrix is stochastic for every $N\ge 1$.
\end{proof}

\subsection{Infinitely many lines}

\label{Section_infinitely_many}

In this section we explain how the transfer matrices $\T^{(N)}$ of the previous section are
related to the measures $\P(b_1,b_2)$.

Let $\W_\infty$ denote the set of infinite \emph{stabilizing} growing sequences of integers
$\x=x_1<x_2<\dots$, i.e.\ such that $x_{i+1}=x_i+1$ for all large enough $i$. For $N=1,2,\dots$,
let $pr_N$ denote the projection $pr_N: \W_\infty\to \W_N$ mapping a sequence to its first $N$
coordinates.

Define the $\W_\infty\times \W_\infty$ matrix $\T^{(\infty)}$ via
\begin{equation}
\label{eq_infinite_transfer_matrix}
 \T^{(\infty)}(\x\to\y;1,a_2,b_1,b_2,c_1,c_2)=\lim_{N\to\infty} \T^{(N)}(pr_N(\x)\to pr_N(\y);
 1,a_2,b_1,b_2,c_1,c_2).
\end{equation}
\begin{lemma}
\label{lemma_T_infty}
 If the weights of the six--vertex model are such that $a_1=a_2=1$ and $b_1<1$, then the
 limit in \eqref{eq_infinite_transfer_matrix} exists. Moreover, $\T^{(\infty)}(\x\to\y;1,1,b_1,b_2,c_1,c_2)$
 vanishes unless $y_i=x_i+1$ for all large enough $i$.
\end{lemma}
\begin{proof} Using the definition \eqref{eq_def_T} of the matrices $\T^{(N)}$, we see that $\T^{(\infty)}$ is given by the
product of weights of the vertices in the one-row configuration. Since $\x\in\W_\infty$, there are
only finitely many vertices of types $a_1$, $b_2$, $c_1$, $c_2$ and therefore, the right-hand side
in \eqref{eq_infinite_transfer_matrix} is non-increasing for large $N$ and the limit in
\eqref{eq_infinite_transfer_matrix} exists. For this limit to be nonzero, we should have
$\T^{(N)}(pr_N(\x)\to pr_N(\y); 1,a_2,b_1,b_2,c_1,c_2)>0$ for all $N=1,2,\dots$ and also the
number of vertices of type $b_2$ should be finite. These two conditions imply $\y\in\W_\infty$ and
$y_i=x_i+1$ for all large enough $i$.
\end{proof}

As in Section \ref{Section_P_as_interacting}, we identify a $\P(b_1,b_2)$--random configuration of
the six--vertex model in the quadrant with a sequence $\x^{b_1,b_2}(t)$, $t=0,1,2,\dots$ of
elements of $\W_\infty$, such that $\x^{b_1,b_2}(t)$ encodes the positions of particles on the
horizontal line $y=t+1/2$. Now the definitions imply the following statement.

\begin{proposition}
\label{Proposition_T_infty}
 $\x^{b_1,b_2}(t)$ is a Markov chain with countable state space $\W_\infty$,
  transitional probabilities $\T^{(\infty)}(\x\to\y;1,1, b_1, b_2,1- b_1,1-
 b_2)$ and initial state  $\x^{b_1,b_2}(0)=(1,2,3,\dots)$.
\end{proposition}

\subsection{Eigenvectors of transfer matrix}

An important property of transfer matrices $\T^{(N)}$ is the following eigenrelation.

\begin{theorem} \label{theorem_eigenrelation}
Fix an integer $N>0$ and $N$ small complex numbers $z_1,\dots,z_N$, such that $|b_2 z_i|<1$ for
$1\le i \le N$ and $1-\frac{a_2+b_1b_2-c_1c_2}{b_1} z_{j} + \frac{a_2 b_2}{b_1} z_{i} z_{j}\ne 0$
for $1\le i<j \le N$. For a permutation $\sigma\in {\mathfrak S}(N)$ set
$$
A_\sigma= (-1)^\sigma \prod_{i<j} \frac{1-\frac{a_2+b_1b_2-c_1 c_2}{b_1} z_{\sigma(i)} + \frac{a_2
b_2}{b_1} z_{\sigma(i)} z_{\sigma(j)}}{1-\frac{a_2+b_1b_2-c_1c_2}{b_1} z_{i} + \frac{a_2 b_2}{b_1}
z_{i} z_{j}}.
$$

Then the function (of $\x=(x_1<\dots<x_N)\in \W_N$)
$$
 \Psi(x_1,\dots,x_N; z_1,\dots,z_N)=\sum_{\sigma\in {\mathfrak S}(N)} A_\sigma \prod_{i=1}^N z_{\sigma(i)}^{x_i}
$$
is an eigenfunction of the transfer matrix  $\T^{(N)}(\x\to\y; 1,a_2,b_1,b_2,c_1,c_2)$ of the
six--vertex model, that is:
\begin{multline}
 \sum_{\y\in\W_N}  \T^{(N)}(\x\to\y; 1,a_2,b_1,b_2,c_1,c_2) \Psi(\y; z_1,\dots,z_N)\\= \left(\prod_{i=1}^N
 \frac{b_1+(c_1c_2-b_1b_2)z_i}{1-b_2z_i}\right) \Psi(\x; z_1,\dots,z_N).
\end{multline}
\end{theorem}
\begin{proof}
A version of this result, when all $a_i,b_i,c_i$ parameters are fixed to be $1$ is given in
\cite{Lieb} (see also \cite{Nolden} for a detailed exposition of a more general case). We follow
the general approach therein, though work directly on $\Z$ (i.e.\ not on $\Z/n\Z$), thus avoiding
dealing with the Bethe equations. The periodic boundary condition asymmetric transfer matrix is
also diagonalized in \cite{JS} via the algebraic Bethe ansatz. A careful translation of that
result into coordinate form and infinite-volume limit should yield an alternative approach (than
that which we take below) to proving this result. Yet another approach is to start from the
formula for the case $a_1=a_2$, $b_1=b_2$, $c_1=c_2$, which can be found e.g.\ in \cite[Section
8]{Baxter} and then reduce a general case to it by conjugations of the transfer matrix and
multiplications by constants.

\medskip

For $1\leq i \leq N$ define
\begin{multline}
\label{eq_x12} g_i(x_i,\ldots,x_N;z_i,\ldots, z_N) \\= \sum_{y_i,\ldots, y_N}
\T^{(N-i+1)}\big((x_i,\ldots,x_N)\to (y_i,\ldots, y_N);1,a_2,b_1,b_2,c_1,c_2\big) z_i^{y_i}\cdots
z_N^{y_N}.
\end{multline}
 In order to simplify the notations, here and below we adopt the conventions
 $g_{N+1}\equiv 1$, $g_{N+2}\equiv 0$ and $x_{N+1}\equiv
+\infty$.

 We may perform the summation over $y_i\in \{x_i,x_i+1,\ldots, x_{i+1}\}$ in \eqref{eq_x12} in order to develop a
recursion for the functions $g_i$ in $i$.  There are three cases to consider in summing over $y_i$
and hence we write
\begin{multline*}
g_i(x_i,\ldots,x_N;z_i,\ldots, z_N)  = g^{(1)}_i(x_i,\ldots,x_N;z_i,\ldots,
z_N)+g^{(2)}_i(x_i,\ldots,x_N;z_i,\ldots, z_N)\\+g^{(3)}_i(x_i,\ldots,x_N;z_i,\ldots, z_N).
\end{multline*}
In case (1) we have $y_i=x_i$ and  find that
\begin{equation}\label{g1}
g^{(1)}_i(x_i,\ldots,x_N) = b_1 z_i^{x_i} g_{i+1}(x_{i+1},\ldots, x_N;z_{i+1},\ldots, z_N).
\end{equation}
The factor $b_1$ came from the $b_1$ vertex at position $x_i$. Case (2) involves summing over
$y_i\in \{x_i+1,\ldots, x_{i+1}-1\}$ and keeping track of the geometric sum of the associated
weights yields
\begin{multline}\label{g2}
g^{(2)}_i(x_i,\ldots,x_N;z_{i},\ldots, z_N) \\= \frac{c_1 c_2 z_i}{1-b_2 z_i} \big(z_i^{x_i} -
b_2^{x_{i+1}-x_i-1} z_i^{x_{i+1}-1}\big)  g_{i+1}(x_{i+1},\ldots, x_N;z_{i+1},\ldots, z_N).
\end{multline}
Note that the condition that $|b_2z_N|<1$ enables us to perform the geometric summations also for
$i=N$ leading to the same formula \eqref{g2} with $b_2^{x_{N+1}} z_N^{x_{N+1}}$ being understood
as $0$ (which agrees with our convention $x_{N+1}=+\infty$).

In case (3) we have $y_{i}=x_{i+1}$. This means that when we subsequently sum over $y_{i+1}$, we
cannot allow the term $y_{i+1}= x_{i+1}$. This implies that
\begin{eqnarray*}
&&g^{(3)}_i(x_i,\ldots,x_N;z_{i},\ldots, z_N) \\
&&= a_2 b_2^{x_{i+1}-x_i-1} z_{i}^{x_{i+1}}\big(g_{i+1}(x_{i+1},\ldots,x_N;z_{i+1},\ldots, z_N)-g^{(1)}_{i+1}(x_{i+1},\ldots,x_N;z_{i+1},\ldots, z_N)\big).
\end{eqnarray*}
Applying (\ref{g1}) to the above, we find that
\begin{eqnarray}\label{g3}
&&g^{(3)}_i(x_i,\ldots,x_N;z_{i},\ldots, z_N) \\
\nonumber&&= a_2 b_2^{x_{i+1}-x_i-1} z_{i}^{x_{i+1}} g_{i+1}(x_{i+1},\ldots,x_N;z_{i+1},\ldots, z_N)\\
\nonumber&&\hskip.2in- a_2 b_1 b_2^{x_{i+1}-x_i-1} \big(z_i z_{i+1}\big)^{x_{i+1}} g_{i+2}(x_{i+2},\ldots,x_N;z_{i+2},\ldots, z_N).
\end{eqnarray}
We again note that due to our conventions $x_{N+1}=+\infty$, $g_{N+2}=0$ and because of
$|b_2z_N|<1$, the formula \eqref{g3} is still valid for $i=N$.

Combining (\ref{g1}),(\ref{g2}), and (\ref{g3}) yields the recursion relation
\begin{eqnarray}\label{rr}
&&g_i(x_i,\ldots,x_N;z_{i},\ldots, z_N) \\
\nonumber&&= \big(L_i z_i^{x_i} + M_i b_2^{x_{i+1}-x_{i}} z_i^{x_{i+1}}\big) g_{i+1}(x_{i+1},\ldots, x_N;z_{i+1},\ldots, z_N)\\
\nonumber&&\hskip.2in - \frac{a_2b_1}{b_2} b_2^{x_{i+1}-x_i} \big(z_i z_{i+1}\big)^{x_{i+1}} g_{i+2}(x_{i+2},\ldots, x_N;z_{i+2},\ldots, z_N),
\end{eqnarray}
where
\begin{equation}\label{LMs}
L_i = \frac{b_1 + (c_1c_2 - b_1b_2)z_i}{1-b_2 z_i}, \qquad \textrm{and} \qquad M_i =
\frac{a_2-c_1c_2 - a_2 b_2 z_i}{b_2(1-b_2 z_i)}.
\end{equation}
This, along with the boundary conditions $g_{N+1}\equiv 1$, $g_{N+2}\equiv 0$ and $x_{N+1}\equiv
+\infty$ determines the value of all $g_i$.

The above recursion shows that $z_1^{y_1}\cdots z_N^{y_N}$ is not an eigenfunction for the
transfer matrix. However, since it provides such a relatively explicit formula for the action of
the transfer matrix on such monomials, one might hope that for suitably chosen coefficients
$S_{\sigma}$ (independent of the $y_i$'s but possibly dependent on the $z_i$'s) the sum
$\sum_{\sigma\in{\mathfrak S}(N)} S_\sigma z_{\sigma(1)}^{y_1}\cdots z_{\sigma(N)}^{y_N}$ may be
an eigenfunction of the transfer matrix. Let us see how this is born out in the example of $N=2$
before going to the general case. The recursion implies
$$
g_1(x_1,x_2;z_1,z_2) = \big(L_1 z_1^{x_1} + M_1 b_2^{x_2-x_1} z_1^{x_2}\big) L_2 z_2^{x_2} - \frac{a_2b_1}{b_2} b_2^{x_{2}-x_1} \big(z_1 z_{2}\big)^{x_{2}}.
$$
Therefore we find that
\begin{multline}
\label{eq_x15}
\sum_{y_1,y_2} \T^{(2)}\big((x_1,x_2)\to (y_1,y_2);1,a_1,b_1,b_2,c_1,c_2\big) \big(S_{12} z_1^{x_1}z_2^{x_2} + S_{21} z_2^{x_1}z_1^{x_2}\big)\\
=S_{12}g_1(x_1,x_2;z_1,z_2) + S_{21} g_1(x_1,x_2;z_2,z_1) = L_1 L_2\big(S_{12} z_1^{x_1}z_2^{x_2}
+ S_{21} z_2^{x_1}z_1^{x_2}\big) \\ + \left(S_{12} \left(M_1 L_2 - \frac{a_2b_1}{b_2}\right)+
S_{21} \left(M_2 L_1 - \frac{a_2b_1}{b_2}\right)\right) b_2^{x_2-x_1}(z_1z_2)^{x_2}.
\end{multline}
Here we have used one-line notation for permutations. If only the last term above were zero, then
we would have an eigenfunction (with eigenvalue $L_1L_2$). Observe that
$$
M_1 L_2 - \frac{a_2b_1}{b_2} = \frac{-c_1 c_2}{b_2(1-b_2z_1)(1-b_2z_2)} s_{12}
$$
where
\begin{equation}\label{smatr}
s_{ij} = b_1 - (a_2+b_1 b_2-c_1 c_2) z_j + a_2b_2 z_i z_j,
\end{equation}
and the rest of the expression is symmetric in $z_1$ and $z_2$. It readily follows that last line
\eqref{eq_x15} is zero if $S_{12} = s_{21}$ and $S_{21} = -s_{12}$. This eigenfunction corresponds
to the claimed one from the theorem, up to scaling by an overall function of the $z_i$'s (which
does not involve the $y_i$'s and hence has no effect).

In order to proceed with the general $N$ proof, it is convenient to develop an expansion for
$g_1(x_1,\ldots,x_N;z_1,\ldots,z_N)$ coming from the recursion relation \eqref{rr}. Let $W$ be the
set of $N$-letter words $w$ composed of $L$, $M$ or the pair $DD$ subject to the condition that
that the $N$-th letter $w_N\neq M$. Some examples of word $w\in W$ are $LLMDD$ or $LDDDD$, though
$LDMLL$, $LDDDL$ and $MLDDM$ are not in $W$. The weight of a word $w\in W$ is defined as the
product of weights corresponding to each letter $L$, $M$, or the pair $DD$ in $w$. The weight of
letter $L$ in position $i$ is $L_i z_i^{x_i}$, the weight of letter $M$ in position $i$ is $M_i
b_2^{x_{i+1}-x_i}z_i^{x_{i+1}}$ and the weight of the pair $DD$ in positions $i$ and $i+1$ is
$-\frac{a_2b_1}{b_2} b_2^{x_{i+1}-x_{i}} \big(z_i z_{i+1}\big)^{x_{i+1}}$. We write $WT(w)$ for
this weight. For example, if $w=LMDDL$ then
$$WT(w) = L_1 z_1^{x_1} \cdot M_2 b_2^{x_{3}-x_2}z_2^{x_{3}} \cdot \Big(-\frac{a_2b_1}{b_2} b_2^{x_{4}-x_{3}} \big(z_3 z_{4}\big)^{x_{4}}\Big) L_5 z_5^{x_5}.$$
With this notation, the recursion relation \eqref{rr} implies that
$$g_1(x_1,\ldots, x_N;z_1,\ldots, z_N) = \sum_{w\in W} WT(w).$$
Let us fix two additional pieces of notation. For a permutation $\sigma$, let $WT_{\sigma}(w)$ be
the weight of $w$ where all $z_1,\ldots,z_N$ are replaced by $z_{\sigma(1)},\ldots,
z_{\sigma(N)}$, and for $1\leq i<N$ let $WT^{(i,i+1)}(w)$ be the weight of $w$ where those terms
corresponding to positions $i$ and $i+1$ are excluded from the product (this assumes that $DD$ is
not located in positions $i-1$ and $i$, or $i+1$ and $i+2$, in which case such exclusion is not
well-defined).

Consider the sum
$$
\sum_{\sigma\in{\mathfrak S}(N)} S_{\sigma} g_1(x_1,\ldots,x_N;z_1,\ldots,z_N),\qquad
\textrm{with}\qquad S_{\sigma} = (-1)^{\sigma} \prod_{i<j} s_{\sigma(j)\sigma(i)}.
$$
We will see that this sum is equal to $\sum_{\sigma} S_{\sigma} WT_{\sigma}(L^N)$. Since
$A_\sigma=\frac{S_\sigma}{S_{Id}}$, this clearly will prove Theorem \ref{theorem_eigenrelation}.

In order to show that only the word $w=L^N$ contributes to this sum, it is useful to split $W$
into three disjoint sets $W_1$, $W_2$ and $W_3$. The set $W_1$ is all words $w\in W$ which contain
$DD$ before any occurrence of the string $ML$; the set $W_2$ is all words $w\in W$ which contain
$ML$ before any occurrence of the string $DD$; and the set $W_3$ is all other words. It is easy to
see that $W_3$ contains exactly one element, namely $L^N$. This is because $w\in W_3$ must not
contain the string $DD$ or $ML$. Thus, it must take the form $L^k M^{N-k}$, but since words in $W$
cannot end with $M$, we must have $k=N$. The sets $W_1$ and $W_2$ are in bijection with each
other. This is realized by replacing the first occurrence of the string $DD$ in a word $w\in W_1$
by $ML$, and likewise replacing the first occurrence of the string $ML$ in a word $w\in W_2$ by
$DD$. For $w\in W_1$ let $\iota w\in W_2$ represent the corresponding word.

Thus, we have
\begin{multline*}
\sum_{\sigma\in{\mathfrak S}(N)} S_{\sigma} g_1(x_1,\ldots,x_N;z_1,\ldots,z_N) =  \sum_{w\in W} \sum_{\sigma\in{\mathfrak S}(N)} S_{\sigma} WT_{\sigma}(w)\\
=  \sum_{\sigma\in{\mathfrak S}(N)} S_{\sigma} WT_{\sigma}(L^N) + \sum_{w\in
W_1}\sum_{\sigma\in{\mathfrak S}(N)} S_{\sigma} \big(WT_{\sigma}(w)+WT_{\sigma}(\iota w)\big).
\end{multline*}
Fix $w\in W_1$ and let $i,i+1$ be the location of the first occurrence of the string $DD$ in $w$.
Consider the inner sum from above
\begin{multline*}
\sum_{\sigma\in{\mathfrak S}(N)} S_{\sigma} \big(WT_{\sigma}(w)+WT_{\sigma}(\iota w)\big) =
\sum_{\sigma\in{\mathfrak S}(N)}(-1)^{\sigma}\prod_{\substack{k<\ell\\(k,\ell)\neq (i,i+1)}}
s_{\sigma(\ell)\sigma(k)} WT_{\sigma}^{(i,i+1)}(w) \\ \times
\left(s_{\sigma(i)\sigma(i+1)}\left(M_{\sigma(i)}L_{\sigma(i+1)} -
\frac{a_1b_1}{b_2}\right)\right)
b_2^{x_{i+1}-x_i}\big(z_{\sigma(i)}z_{\sigma(i+1)}\big)^{x_{i+1}}.
\end{multline*}
Replacing $\sigma$ by $(i,i+1)\sigma$ (where $(i,i+1)$ is the transposition of $i$ and $i+1$) does
not change the value of the above sum. This transposition only changes the term
$s_{\sigma(i)\sigma(i+1)}\big(M_{\sigma(i)}L_{\sigma(i+1)} - \frac{a_1b_1}{b_2}\big)$ to
$s_{\sigma(i+1)\sigma(i)}\big(M_{\sigma(i+1)}L_{\sigma(i)} - \frac{a_1b_1}{b_2}\big)$ and changes
the sign of $(-1)^{\sigma}$ in the above expression. Therefore
\begin{multline*}
\sum_{\sigma\in{\mathfrak S}(N)} S_{\sigma} \big(WT_{\sigma}(w)+WT_{\sigma}(\iota w)\big) =
\frac{1}{2} \sum_{\sigma\in{\mathfrak S}(N)}(-1)^{\sigma}\prod_{\substack{k<\ell\\(k,\ell)\neq
(i,i+1)}} s_{\sigma(\ell)\sigma(k)}  WT_{\sigma}^{(i,i+1)}(w) b_2^{x_{i+1}-x_i}\\ \times
\big(z_{\sigma(i)}z_{\sigma(i+1)}\big)^{x_{i+1}}
\left(s_{\sigma(i)\sigma(i+1)}\left(M_{\sigma(i)}L_{\sigma(i+1)} - \frac{a_1b_1}{b_2}\right)-
s_{\sigma(i+1)\sigma(i)}\left(M_{\sigma(i+1)}L_{\sigma(i)} - \frac{a_1b_1}{b_2}\right)\right) .
\end{multline*}
However, as we have already observed in the case $N=2$, the expression
$$
s_{\sigma(i)\sigma(i+1)}\left(M_{\sigma(i)}L_{\sigma(i+1)} - \frac{a_1b_1}{b_2}\right)-
s_{\sigma(i+1)\sigma(i)}\left(M_{\sigma(i+1)}L_{\sigma(i)} - \frac{a_1b_1}{b_2}\right) =0
$$
by virtue of the definition of $s_{ij}$ in (\ref{smatr}). This shows that
$$
\sum_{\sigma\in{\mathfrak S}(N)} S_{\sigma} g_1(x_1,\ldots,x_N;z_1,\ldots,z_N) =
\sum_{\sigma\in{\mathfrak S}(N)} S_{\sigma} WT_{\sigma}(L^N),
$$
which implies the claimed result of Theorem \ref{theorem_eigenrelation}.
\end{proof}

\begin{corollary}
\label{corollary_eigenrelation_T} Fix an integer $N>0$ and complex numbers $z_1,\dots,z_N$, such
that $|b_2 z_i|<1$ for $1\le i \le N$ and $1-\frac{a_2+b_1b_2-c_1c_2}{b_1} z_{j} + \frac{a_2
b_2}{b_1} z_{i} z_{j}\ne 0$ for $1\le i<j \le N$. For a permutation $\sigma\in {\mathfrak S}(N)$
set
$$
A'_\sigma= (-1)^\sigma \prod_{i<j} \frac{1-\frac{a_2+b_1b_2-c_1 c_2}{b_1} z_{\sigma(j)} + \frac{a_2
b_2}{b_1} z_{\sigma(i)} z_{\sigma(j)}}{1-\frac{a_2+b_1b_2-c_1c_2}{b_1} z_{j} + \frac{a_2 b_2}{b_1}
z_{i} z_{j}}.
$$
Then the function (of $X=(x_1<\dots<x_N)\in\W_N$)
$$
 \Phi(x_1,\dots,x_N; z_1,\dots,z_N)=\sum_{\sigma\in{\mathfrak S}(N)} A'_\sigma \prod_{i=1}^N z_{\sigma(i)}^{-x_i}
$$
is an eigenfunction of the \emph{transposed} transfer matrix of the six--vertex model with
parameters $a_1=1,a_2,$ $b_1,b_2$, $c_1,c_2$, that is:
\begin{multline*}
 \sum_{\x\in\W_N} \T^{(N)}(\x\to\y; 1,a_2,b_1,b_2,c_1,c_2) \Phi(\x; z_1,\dots,z_N)\\= \left(\prod_{i=1}^N
 \frac{b_1+(c_1c_2-b_1b_2)z_i}{1-b_2z_i}\right) \Phi(\y; z_1,\dots,z_N).
\end{multline*}
\end{corollary}
\begin{proof}
The transposed transfer matrix differs from the original transfer matrix by the horizontal flip
 $\x\mapsto -\x$, $\y\mapsto -\y$, by the vertical flip which interchanges the roles  $c_1\leftrightarrow
 c_2$ and by the change of order of the coordinates (the latter is responsible for the change $i\leftrightarrow j$ between
 $A_\sigma$ and $A'_\sigma$). It remains to observe that the swap $c_1\leftrightarrow
 c_2$ does not change the transfer matrix.
\end{proof}

\subsection{Spectral decomposition of transfer matrix}

The eigenrelation implies integral formulas for the transfer matrix. Let us introduce the notation
$\T^{(N)}_t$ for the $t$--th power of the matrix $\T^{(N)}$. In particular, $\T^{(N)}_1=\T^{(N)}$.

\begin{theorem} \label{Theorem_eigentransfer} Fix $N=1,2,3,\dots$ and define for $\sigma\in
{\mathfrak S}(N)$
$$
A'_\sigma= (-1)^\sigma \prod_{i<j} \frac{1-\frac{a_2+b_1b_2-c_1 c_2}{b_1} z_{\sigma(j)} + \frac{a_2
b_2}{b_1} z_{\sigma(i)} z_{\sigma(j)}}{1-\frac{a_2+b_1b_2-c_1c_2}{b_1} z_{j} + \frac{a_2 b_2}{b_1}
z_{i} z_{j}},
$$
and
$$
 A''_\sigma= (-1)^\sigma \prod_{i<j} \frac{1-\frac{a_2+b_1b_2-c_1 c_2}{a_2 b_2} z_{\sigma(i)} + \frac{b_1}{a_2
b_2} z_{\sigma(i)} z_{\sigma(j)} }{1 -\frac{a_2+b_1b_2-c_1c_2}{a_2b_2} z_{i} + \frac{b_1}{a_2 b_2}
z_{i} z_{j} }.
$$
(The second one is obtained from the first one by the inversion $z_k\to z_k^{-1}$, $k=1,\dots,N$).
We have the following decompositions for the $t$-th power of the transfer matrix
\begin{multline}
\label{eq_transfer_non_TW_form} \T^{(N)}_t(\y\to\x; 1,a_2,b_1,b_2,c_1,c_2)\\=
\frac{1}{(2\pi\ii)^N} \int_{(C_r)^N} \sum_{\sigma\in {\mathfrak S}(N)} A'_\sigma \prod_{i=1}^N
z_{\sigma(i)}^{-x_i} z_i^{y_i-1}\left( \frac{b_1+(c_1c_2-b_1b_2)z_i}{1-b_2z_i}\right)^t dz_1\cdots
 dz_N
\end{multline}
and
\begin{multline}
\label{eq_transfer_TW_form}
 \T^{(N)}_t(\y\to\x ; 1,a_2,b_1,b_2,c_1,c_2)\\=  \frac{1}{(2\pi\ii)^N} \int_{(C_R)^N} \sum_{\sigma\in
{\mathfrak S}(N)} A''_\sigma \prod_{i=1}^N z_{\sigma(i)}^{x_i} z_i^{-y_i-1}\left(
\frac{b_1+(c_1c_2-b_1b_2)z_i^{-1}}{1-b_2z_i^{-1}}\right)^t  dz_1\cdots
 dz_N,
\end{multline}
where $C_r$ is a small positively oriented circle around origin (i.e.\ such that all non-zero
singularities of the integrand lie outside) and $C_R$ is a large positively oriented circle (i.e.\
such that all singularities of the integrand lie inside).
\end{theorem}
\noindent{\bf Remark.} The formula \eqref{eq_transfer_TW_form} is similar to the contour integral
representation for the transitional probability of ASEP, see \cite[Theorem 2.1]{TW_determ}.
\begin{proof}
 We start with the following identity for $K=(k_1<\dots<k_N)$ and $L=(l_1<\dots<l_N)$: Set as in
 Theorem \ref{theorem_eigenrelation}
 $$
 A_\sigma= (-1)^\sigma \prod_{i<j} \frac{1-\frac{a_2+b_1b_2-c_1 c_2}{b_1} z_{\sigma(i)} + \frac{a_2
b_2}{b_1} z_{\sigma(i)} z_{\sigma(j)}}{1-\frac{a_2+b_1b_2-c_1c_2}{b_1} z_{i} + \frac{a_2 b_2}{b_1}
z_{i} z_{j}}.
$$
(note the difference both with $A'_\sigma$ and $A''_\sigma$). Then
\begin{equation}
\label{eq_init_cond_identity}
 \frac{1}{(2\pi\ii)^N}\int_{(C_r)^N} \sum_{\sigma\in \S_N} A_\sigma \prod_{i=1}^N z_{\sigma(i)}^{k_i}  z_i^{-l_i-1} dz_1\cdots
 dz_N={\mathbf 1}_{K=L},
\end{equation}
where the integration goes over the contours $C_r$ which are positively oriented circles around
the origin of a very small radius, and ${\mathbf 1}_{K=L}$ is the indicator function of $K=L$.
This identity under the assertion of equation \eqref{eq_quadratic_condition} was proved in
\cite[Section II.4.b]{TW_review}, \cite[Theorem 2.1]{TW_determ}  (see also \cite[Section
3]{BCPS2}); the general weight case is obtained by the change of variables of the form $z_i\mapsto
\theta z_i$, $i=1,\dots,N$, for suitable $\theta$.

Set $\hat K=(-k_N<\dots<-k_1)=(\hat k_1<\hat k_2<\dots<\hat k_N)$ and $\hat
L=(-l_N<\dots<l_1)=(\hat l_1<\hat l_2<\dots<\hat l_N)$, then \eqref{eq_init_cond_identity} yields
\begin{equation}
\label{eq_init_cond_identity_2}
 \frac{1}{(2\pi\ii)^N}\int_{(C_r)^N} \sum_{\sigma\in\S_N} A_\sigma \prod_{i=1}^N z_{\sigma(i)}^{-\hat k_{N+1-i}}  z_i^{\hat l_{N+1-i}-1} dz_1\cdots
 dz_N={\mathbf 1}_{K=L}.
\end{equation}
Make the change of index $i\mapsto (N+1-i)$ in \eqref{eq_init_cond_identity_2}. In $A_\sigma$ this
turns $\prod_{i<j}$ into $\prod_{i>j}$. Thus, $A$ turns into $A'$ and we get
\begin{equation}
\label{eq_init_cond_identity_3} \frac{1}{(2\pi\ii)^N} \int_{(C_r)^N} \sum_{\sigma\in\S_N}
A'_\sigma \prod_{i=1}^N z_{\sigma(i)}^{-\hat k_{i}}  z_i^{\hat l_{i}-1} dz_1\cdots
 dz_N=\delta(\hat K,\hat L).
\end{equation}
Multiply \eqref{eq_init_cond_identity_3} by the matrix $\T^{(N)}_t(\hat K \to X)$ and sum over
$\hat K$ using the result of Corollary \ref{corollary_eigenrelation_T}. We get
\begin{equation}
  \frac{1}{(2\pi\ii)^N}\int_{(C_r)^N} \sum_{\sigma\in\S_N} A'_\sigma \prod_{i=1}^N z_{\sigma(i)}^{-x_i}  z_i^{\hat l_i-1}\left( \frac{b_1+(c_1c_2-b_1b_2)z_i}{1-b_2z_i}\right)^t  dz_1\cdots
 dz_N=\T^{(N)}_t(\hat L\to X)
\end{equation}
Renaming $\hat L$ into $Y$ we get \eqref{eq_transfer_non_TW_form}. In order to get
\eqref{eq_transfer_TW_form} we make a change of variables $z_i\mapsto z_i^{-1}$ in
\eqref{eq_transfer_non_TW_form}.
\end{proof}

\section{Marginals and observables}

The aim of this section is to produce concise formulas for expectations of some observables with
respect to the probability distribution obtained by application of powers  of the transfer matrix
$\T^{(\infty)}$ to the step initial condition. These ultimately lead to formulas for  one--point
marginal distributions of the configuration of the six--vertex model at row $y=t$, which are
suitable for $t\to\infty$ asymptotic analysis. Our present approach unites and generalizes the
approaches developed in \cite{TW_determ} and \cite{BigMac}, \cite{BCS}.

As is explained in Section \ref{Section_infinitely_many}, the information about $\P(b_1,b_2)$ can
be extracted from ($N\to\infty$ limit of) transfer matrices $\T^{(N)}(\,\cdot\,;
1,1,b_1,b_2,1-b_1,1-b_2)$. The starting point for the arguments of this section is the
representation for these matrices of Theorem \ref{Theorem_eigentransfer}, which for convenience we
reproduce here with the new notation $\tau=b_2/b_1$. Later we will also assume that $\tau<1$.

\begin{corollary} \label{corollary_stochastic_transfer}Fix a positive integer $N$. With the notation $\tau=b_2/b_1$ define
$$
 A''_\sigma= (-1)^\sigma \prod_{i<j} \frac{1-(1+\tau^{-1})z_{\sigma(i)} + \tau^{-1} z_{\sigma(i)} z_{\sigma(j)} }{1 -(1+\tau^{-1}) z_{i} +
 \tau^{-1}
z_{i} z_{j} }.
$$
 We have the
following decompositions for the $t$-th power of the transfer matrix $\T^{(N)}$
\begin{multline}
\label{eq_transfer_TW_form_s}
 \T^{(N)}_t(\y\to\x; 1,1,b_1,b_2,1-b_1,1-b_2)\\= \frac{1}{(2\pi\ii)^N} \int_{(C_R)^N} \sum_{\sigma\in {\mathfrak S}(N)} A''_\sigma \prod_{i=1}^N z_{\sigma(i)}^{x_i}
z_i^{-y_i-1}\left( \frac{b_1+(1-b_1-b_2)z_i^{-1}}{1-b_2z_i^{-1}}\right)^t  dz_1\cdots
 dz_N,
\end{multline}
where $C_R$ is a large positively oriented circle around origin containing all singularities of
the integrand, and $\x=(x_1<\dots<x_n)$, $\y=(y_1<\dots<y_n)$.
\end{corollary}

\subsection{Contour deformations}

The following lemma turns out to be crucial for the analysis. We are going to use it with
\begin{equation}
\label{eq_f_definition}
 f(z)=\left( \frac{b_1+(1-b_1-b_2)z^{-1}}{1-b_2z^{-1}}\right)^t,
\end{equation}
as in Corollary \ref{corollary_stochastic_transfer}.

\begin{lemma}
\label{lemma_deformation} Consider the $N$--dimensional contour integral
\begin{equation}
\label{eq_predeformed}
 \int_{(C_R)^N} \sum_{\sigma\in {\mathfrak S}(N)} A''_\sigma \cdot \prod_{i=1}^N z_{\sigma(i)}^{x_i}
z_i^{-y_i-1} f(z_i) \cdot G\left(z_{\sigma(1)},\dots,z_{\sigma(N)}\right)\, dz_1\cdots
 dz_N,
 \end{equation}
 where $C_R$ is a positively oriented circle of a large radius $R$ (containing all singularities of the
 integrand),
 $A''_{\sigma}$ is given in Corollary
 \ref{corollary_stochastic_transfer}, and $x_1\le x_2\le\dots\le x_N$, $y_1\le y_2\le \dots\le y_N$  are two arbitrary ordered
 sequences of integers. Suppose that $f$ is a meromorphic function with a single
 singularity at a point $\mathfrak s\in\mathbb C$ such that $s:=|\mathfrak s|<1$ and such that the function
 $$
  f(z)f\left(1+\tau -\tau z^{-1} \right)
 $$
 has no singularity at $\mathfrak s$. Further assume that there exists a circular contour $C_{r'}$ such that:
 \begin{enumerate}
  \item $s<r'<1$
  \item $\left|1+\tau -\tau z^{-1} \right|>1$ for $z$ inside $C_{r'}$
  \item $|-(1+\tau^{-1}) z +
 \tau^{-1}
z z'|<1$ for $z$, $z'$ inside $C_{r'}$.
 \end{enumerate}
 Also assume that $G(z_1,\dots,z_N)$ is a holomorphic function of $z_1,\dots,z_N$ (without
 singularities).

 Then for any $k\in\{1,\dots,N\}$, we can deform the $k$ first contours (i.e.\ variables $z_1,\dots,z_k$) in
 \eqref{eq_predeformed} without changing the
 value of the integral,  from
 large contours $C_R$ to medium size contours $C_{r'}$.
\end{lemma}
\noindent{\bf Remark 1.} The existence  of the medium size contour $C_{r'}$  is assured as long as
$s$ is sufficiently small. In our application of this this lemma to  the function $f$ in
\eqref{eq_f_definition}, as long as $b_1$ and $b_2$ are small, the corresponding $s$ will be as
well. Later this restriction on the $b_1$, $b_2$ will be relaxed by use of an analytic
continuation argument.

\smallskip

\noindent{\bf Remark 2.} A somewhat similar statement can be found in \cite[Lemma 5.1]{TW_determ}.

\begin{proof}[Proof of Lemma \ref{lemma_deformation}]
The proof is induction in $k$. We start with $k=1$ and deform the $z_1$ contour. When we deform
it, we might encounter poles from the denominator factors
$$ 1 -(1+\tau^{-1}) z_{1} +
 \tau^{-1}
z_{1} z_{j}=0, \quad j=2,\dots,N,
$$
which arise when $z_1= ((1+\tau^{-1})-\tau^{-1} z_j)^{-1}$. Since the radius $R$ of the $C_R$
contour (along which  the $z_j$s are integrated) is large, and the radius $r'$ is bounded away
from $0$, we find that these possible poles in $z_1$ are within $C_{r'}$. Hence, we may freely
deform $z_1$ from $C_R$ to $C_{r'}$.

\smallskip

Now suppose that the first $k-1$ contours (i.e.\ variables $z_1,\dots,z_{k-1}$) are along $C_{r'}$
 and let us deform the $z_k$-contour. Before doing that
is is convenient to slightly change the contours so that $z_i$, $i=1,\dots,k-1$, are integrated
over $C_{r'-i \eps}$ (with $\eps$ small) and $z_j$, $j=k+1,\dots,N$, are integrated over
$C_{R-j\eps}$. Clearly, such deformation will not change the integrals and after we deform
$z_k$-contour from $C_R$ to $C_{r'}$ we can bring all the contours back to $C_{r'}$ and $C_{R}$.

While deforming the $z_k$-contour we might encounter two kinds of possible poles in the
deformation. The first ones are of the same type as before, corresponding to  denominator factors
$$ 1 -(1+\tau^{-1}) z_{k} +
 \tau^{-1}
z_{k} z_{j}=0, \quad j=k+1,\dots,N,
$$
and, as before, these poles are not crossed  during the deformation.

The other set of possible poles correspond to denominator factors
\begin{equation}\label{eq_pole} 1 -(1+\tau^{-1}) z_{i} +
 \tau^{-1}
z_{i} z_{k}=0,\quad i=1,\dots,k-1,
\end{equation}
arise when $z_k= 1+\tau -\tau z_i^{-1}$. Such a pole is only present when
$\sigma^{-1}(i)>\sigma^{-1}(k)$
--- otherwise there is a matching term in the numerator of $A''_\sigma$ which cancels the denominator.
Note that these poles (when they are present) lie between $C_R$ and $C_{r'}$ contours (this is
assured by the definition of $C_{r'}$) and hence we must consider  the residues from them. (Note
that all the poles are distinct and simple because of the small perturbation of the contours that
we made at the beginning of the proof.) In the rest of the proof we show that the total sum of
residues  of such poles vanishes.

Let us fix an index $i$ such that $\sigma^{-1}(i)>\sigma^{-1}(k)$ and compute the residue in
$z_k=1+\tau-\tau z_i^{-1}$. The $z_i$--dependent part of this residue is:

\begin{multline}
\label{eq_residue_form}
 (\cdots) \prod_{j={1}}^{i-1} \left(\frac{1}{1 -(1+\tau^{-1}) z_{j} +
 \tau^{-1}
z_{i} z_{j}} \cdot \frac{1}{1 -(1+\tau^{-1}) z_j +
 \tau^{-1}
 \left( 1+\tau -\tau z_i^{-1}\right)z_{j}} \right)
 \\
\times \prod_{j={i+1}}^{k-1} \left(\frac{1}{1 -(1+\tau^{-1}) z_{i} +
 \tau^{-1}
z_{i} z_{j}} \cdot \frac{1}{1 -(1+\tau^{-1}) \left( 1+\tau -\tau z_i^{-1}\right) +
 \tau^{-1}
 \left( 1+\tau -\tau z_i^{-1}\right)z_{j}} \right)
\\
\times\prod_{j={k+1}}^{N} \left(\frac{1}{1 -(1+\tau^{-1}) z_{i} +
 \tau^{-1}
z_{i} z_{j}} \cdot \frac{1}{1 -(1+\tau^{-1}) \left( 1+\tau -\tau z_i^{-1}\right) +
 \tau^{-1}
 \left( 1+\tau -\tau z_i^{-1}\right)z_{j}} \right)
 \\
  \times z_i^{-y_i-1} \left( 1+\tau -\tau z_i^{-1}\right)^{-y_k-1}
 z_i^{x_{\sigma^{-1}(i)}} \left( 1+\tau -\tau z_i^{-1}\right)^{x_{\sigma^{-1}(j)}}
 f(z_i)
 f\left(1+\tau - \tau z_i^{-1}\right),
\end{multline}
where $(\cdots)$ stands for several additional quadratic (holomorphic) factors arising from the
numerator in the definition of $A''_{\sigma}$.

Recall that $z_i$ is integrated over $C_{r'- i \eps}$ and let us compute the integral of
\eqref{eq_residue_form} over this contour. The integrand has the following (potential) poles as
$z_i$ varies inside $C_{r'}$:
\begin{itemize}
 \item By assumption, the factor $ f(z_i)
 f\left(1+\tau - \tau z_i^{-1}\right)$ has no pole at $z_i=\mathfrak s$ (which is inside $C_{r'}$),
 but might have a pole when $1+\tau -\tau z_i^{-1}=\mathfrak s$. But the  second condition on
 $r'$ guarantees that the latter pole is outside $C_{r'}$.

 \item The total power of $z_i$ in the integrand is $$
 z_i^{(y_k-y_i)+(x_{\sigma^{-1}(i)}-x_{\sigma^{-1}(k)})},
$$
and both terms in the exponent are positive here, since $\sigma^{-1}(i)>\sigma^{-1}(k)$ and $x$'s
and $y$'s  are ordered. Thus, there is no pole at zero coming from this part.

\item The factor $\left( 1+\tau - \tau z_i^{-1}\right)$ is non-zero inside $C_{r'}$.

 \item The first factor in $\prod_{j=1}^{i-1}$ has no poles inside $C_{r'}$,
 which  is guaranteed by the  third condition on $r'$.

 \item
 The
 second factor in $\prod_{j=1}^{i-1}$ equals $(1-z_j z_i^{-1})^{-1}$. Thus, the pole
 appears at $z_i=z_j$. Recall that $z_i$ is integrated over $C_{r'-i\eps}$, while $z_j$ is
 integrated over $C_{r'-j\eps}$. Since $i>j$, the latter contour is larger and thus there is no pole
 at $z_i=z_j$ inside $z_i$--contour.\footnote{Had we chosen the integration contours in a different way, e.g.\
if $\eps<0$, then these poles would give non-zero residues. However, when we sum over permutations
$\sigma$ the contributions would still cancel out, as follows from the argument similar to the one
that we will use further in this proof.}

\item The first factor in $\prod_{j=i+1}^{k-1}$ has no poles inside $C_{r'}$, which
is guaranteed by the third condition on $r'$.

 \item The second factor in $\prod_{j=i+1}^{k-1}$ has no poles inside $C_{r'}$, since $z_j$ is
 integrated over the $C_{r'}$ contour with $r'<1$ and thus
  $\left|\left( 1+\tau -\tau z_i^{-1}\right)(-1-\tau^{-1}  +
 \tau^{-1} z_{j})\right|>\left| 1+\tau -\tau z_i^{-1}\right|>1$ inside $C_{r'}$.

 \item The second factor in $\prod_{j=k+1}^{N}$ has no poles inside $C_{r'}$, since $z_j$ is
 integrated over the large contour $C_R$, and we have  that
 $\left|1+\tau - \tau z_i^{-1}\right|>1$ inside $C_{r'}$.

\end{itemize}

What remains to consider is the first factor in $\prod_{j=k+1}^{N}$. For each such $j$, if
$\sigma^{-1}(j)>\sigma^{-1}(i)$ then this factor cancels by the corresponding term in the
numerator of $A''_\sigma$. On the other hand, if $\sigma^{-1}(j)<\sigma^{-1}(i)$, then this factor
indeed gives a pole inside $C_{r'}$ at $z_i=((1+\tau^{-1})-z_j \tau^{-1})^{-1}$ (note that all
such poles are distinct because of the small perturbation of contours that we did at the beginning
of the proof). Note that the substitution of such $z_i$ turns $(1+\tau) -\tau z_i^{-1}$ into
$z_j$:
$$
 \frac{(1+\tau) z_i
-\tau}{z_i}=((1+\tau^{-1})-z_j \tau^{-1}) \left(\frac{1+\tau}{(1+\tau^{-1})-z_j
\tau^{-1}}-\tau\right)=z_j.
$$
To summarize, we took a pair of indices $i$, $j$ such that $i<k<j$ and a term corresponding to a
permutation $\sigma$ such that $\sigma^{-1}(i)>\sigma^{-1}(k)$, $\sigma^{-1}(i)>\sigma^{-1}(j)$.
Then we first took the residue in $z_k$ at $z_k=(1+\tau) -\tau z_i^{-1}$ and after that took the
residue in $z_i$ at $z_i=((1+\tau^{-1})-\tau^{-1}z_j )^{-1}$. Effectively, this means that we
first multiplied the integrand by $(z_k-((1+\tau)  -\tau z_i^{-1}))$ and by
$(z_i-((1+\tau^{-1})-z_j \tau^{-1})^{-1})$ and then plugged in into the result $z_k=z_j$ and
$z_i=((1+\tau^{-1})-z_j \tau^{-1})^{-1}$. We claim that if we now fix $i<k<j$ and sum over all
$\sigma$, then these residues cancel out because of skew-symmetry, which is seen through the
following argument.

Note that if we multiply the original integrand in \eqref{eq_predeformed} by $\prod_{i<j}\left(1
-(1+\tau^{-1}) z_{i} +
 \tau^{-1}
z_{i} z_{j} \right) \prod_{i=1}^N z_i^{y_i}$ and then do the summation in its definition only over
all permutations $\sigma$ such that $\sigma^{-1}(i)>\sigma^{-1}(k)$,
$\sigma^{-1}(i)>\sigma^{-1}(j)$ (recall that other $\sigma$'s do not contribute), then the result
would be skew--symmetric in the pair of variables $z_k$, $z_j$ and, thus, vanish when $z_k=z_j$.
This is because all such permutations can be split into pairs  with permutations in each pair
different by the transposition $(k,j)$. When we divide back by $\prod_{i<j}\left(1 -(1+\tau^{-1})
z_{i} +
 \tau^{-1}
z_{i} z_{j} \right)\prod_{i=1}^N z_i^{y_i}$ and further multiply by $(z_k-((1+\tau) -\tau
z_i^{-1}))$ and by $(z_i-((1+\tau^{-1})-\tau^{-1}z_j )^{-1})$ the vanishing property is kept
intact.
\end{proof}

\subsection{Notations and summation formulas}

In what follows we use the $q$--algebra notations:
$$
 (x;q)_k=\prod_{i=1}^k (1-x q^{i-1}),\quad {N \choose k}_q = \frac{(1-q^N)(1-q^{N-1})\cdots (1-q^{N-k+1})}{(1-q)(1-q^2)\cdots(1-q^k)} =\frac{
 (q^{N-k+1};q)_k}{(q;q)_k}.
$$
In the first definition $k$ is a non-negative integer or $+\infty$. The following two summation
formulas are known as the $q$--binomial theorems, see e.g.\ \cite[Theorem 10.2.1 and Corollary
10.2.2(c)]{AAR}.
\begin{lemma} \label{Lemma_q_binomial} For any $q,t\in\mathbb C$ such that $|q|\ne 1$ and any positive integer $n$, we have
\begin{equation}
\label{eq_q_binomial}
 \prod_{i=0}^{n-1} (1+q^i t) =\sum_{k=0}^n q^{k(k-1)/2} {n\choose k}_q t^k.
\end{equation}
\end{lemma}

\begin{lemma} \label{Lemma_q_binomial_2} For any $a,x,q\in\mathbb C$ such that $|q|<1$, $|x|<1$
we have
\begin{equation}
\label{eq_q_binomial_3}
 \frac{(ax;q)_\infty}{(x;q)_\infty} =\sum_{k=0}^\infty \frac{(a;q)_k}{(q;q)_k} x^k.
\end{equation}
\end{lemma}

We also need two symmetrization formulas. For a function $G$ of $N$ variables $z_1,\dots,z_N$, we
denote
$$
 Sym_N(G)=\sum_{\sigma\in\mathfrak S(N)} G(z_{\sigma(1)},z_{\sigma(2)},\dots,z_{\sigma(N)}).
$$

\begin{lemma} \label{lemma_TW_identity} Let $\alpha$ be a complex number. Then
$$
 Sym_N\left(\prod_{1\le i<j \le N} \frac{1-(\alpha+1) z_i+\alpha z_i z_j}{z_j-z_i} \prod_{i=1}^N \frac{z_i^{i-1}}{
 1- z_i z_{i+1} \cdots z_N} \right)= \prod_{i=1}^N \frac{1}{1-z_i}.
$$
\end{lemma}
\begin{proof} This identity is due to Tracy and Widom, see e.g.\ \cite[Section 5]{TW_review}.
Another proof can be found in \cite[Section 7.6]{BCPS2}.
\end{proof}
\begin{lemma} \label{Lemma_Hall-littlewood}
Let $\alpha$ be a complex number. Then
$$
 Sym_N\left(\prod_{1\le i<j \le N} \frac{1-(\alpha+1) z_i+\alpha z_i z_j}{z_j-z_i}  \right)=\frac{(\alpha;\alpha)_N}{(1-\alpha)^N}.
$$
\end{lemma}
\begin{proof} The identity \cite[Chapter III, (1.4)]{M} yields
$$
 Sym_N \left(\prod_{1\le i<j\le N} \frac{u_i-\alpha u_j}{u_i-u_j}\right) =
 \frac{(\alpha;\alpha)_N}{(1-\alpha)^N}.
$$
It remains to change the variables
$$
 u_i=\frac{z_i-1}{z_i-\alpha^{-1}},\quad\quad \frac{u_i-\alpha u_j}{u_i-u_j}=
 \frac{1-(1+\alpha)z_i+\alpha z_i z_j}{z_j-z_i}.\qedhere
$$
\end{proof}

\begin{lemma} \label{lemma_another_symmetrization}  Let $\alpha$ be a complex number, then
\begin{multline}
Sym_N\left(\prod_{1\le i<j\le N} \frac{z_j-z_i}{1-(\alpha+1)z_i+\alpha z_i z_j}
\prod_{i=1}^{N}\frac{1}{(\alpha z_i )\cdots(\alpha z_N )-1}  \prod_{i=1}^N (1-\alpha z_i )
\right)\\= Sym_N\left(\frac{(\alpha-1)^N}{(\alpha;\alpha)_N} \prod_{1\le i<j\le N}
\frac{z_j-z_i}{1-(\alpha+1)z_i+\alpha z_i z_j} \right).
\end{multline}
\end{lemma}

\begin{proof} The proof is a combination of Lemma \ref{lemma_TW_identity} and Lemma \ref{Lemma_Hall-littlewood}.
For more details we refer to \cite[Section 7.2]{BCS}, where it is shown that
\begin{multline} \label{eq_x7}
Sym_N\left(\prod_{1\le i<j\le N} \frac{\xi_j-\xi_i}{1-(1+ \tau^{-1})\xi_j+ \tau^{-1}\xi_i \xi_j}
\prod_{i=1}^{N}\frac{1}{(\xi_1)\cdots(\xi_i)-1}  \prod_{i=1}^N (1-\xi_i) \right)\\=
Sym_N\left(\frac{( \tau-1)^N}{(\tau; \tau)_N} \prod_{1\le i<j\le N} \frac{\xi_j-\xi_i}{1-(1+
\tau^{-1})\xi_j+ \tau^{-1}\xi_i\xi_j}
  \right).
\end{multline}
Setting $\tau=\alpha$ and $\xi_i=\alpha z_{k+1-i} $ in \eqref{eq_x7}, we get the desired
statement.
\end{proof}

\begin{lemma} \label{lemma_TW_identity_2} For any positive integers $m<n$ we have
\begin{multline*}
 \sum_{\substack{S\subset \{1,\dots,n\}\\ |S|=m}} \prod_{\substack{i\in S \\  j\in \{1,\dots,n\}\setminus
 S}} \frac{1-(1+\alpha)z_i+\alpha z_i z_j}{z_j-z_i} \left(1-\prod_{j\in \{1,\dots,n\}\setminus S} z_j
 \right)\\= \alpha^{m(n-m)} {n-1 \choose m}_{\alpha^{-1}} \left(1-\prod_{i=1}^n z_i\right).
\end{multline*}
\end{lemma}
\begin{proof} This identity is due to Tracy and Widom, see \cite[(1.9)]{TW_determ}.
\end{proof}

\subsection{Distribution of a single particle}

For a set (i.e.\ event) $\mathcal A$, let $P_\y(\mathcal A;t)$ have the meaning
 $$
  P_\y(\mathcal A;t)=\sum_{\x\in\mathcal A} \T^{(N)}_t(\y\to\x ;1,1,b_1,b_2,1-b_1,1-b_2).
 $$

\begin{theorem} \label{theorem_particle_distribution}
Fix $N>0$, stochastic parameters of the six--vertex model $b_1$, $b_2$, and a positive integer
$t$; set $\tau=b_2/b_1$ Then for $1\le m \le N$ we have
\begin{multline}
\label{eq_Transfer_marginal} P_\y(x_m=x;t)=(-1)^{m-1} \tau^{m(m-1)/2} \sum_{m\le k \le
N}\sum_{|S|=k} \tau^{\kappa(S,\mathbb Z_{>0})-mk-k(k-1)/2} { {k-1} \choose
{m-1}}_{\tau}\\
\times \frac{1}{(2\pi\ii)^k} \oint \dots \oint \prod_{i,j\in S,\, i<j}
\frac{z_j-z_i}{1-(1+\tau^{-1})z_i+\tau^{-1}z_iz_j} \\ \times  \frac{1-\prod_{i\in S}
z_i}{\prod_{i\in S} (1-z_i)} \prod_{i\in S} z_i^{x-y_i-1}\left(
\frac{b_1+(1-b_1-b_2)z_i^{-1}}{1-b_2z_i^{-1}}\right)^t dz_i,
\end{multline}
where the summation goes over $S\subset\{1,2,\dots,N\}$ of size $k$, $\kappa(S,\mathbb Z_{>0})$ is
the sum of elements in $S$, and contours are positively oriented large circles of equal radius
which contain all singularities of the integrand.
\end{theorem}
\noindent{\bf Remark.}
 The proof which we present below closely follows a similar proof of Tracy and Widom in the context of $N$--particle
ASEP, see \cite[Section 6]{TW_review}, \cite{TW_determ}.

The three key technical ingredients of the proof are Lemmas \ref{lemma_deformation},
\ref{lemma_TW_identity}, \ref{lemma_TW_identity_2}. Note that when we go outside stochastic
restriction \eqref{eq_GS_condition}, the quadratic cross-term in \eqref{eq_transfer_TW_form_s}
changes and it is unclear how to produce an analogue of Lemma \ref{lemma_TW_identity}.

It is, perhaps, natural to try to implement the approach developed in the study of ASEP for
stochastic six--vertex model, since ASEP itself can be viewed as a small $b_1$, $b_2$ limit of the
stochastic six--vertex model, see Section \ref{Section_P_as_interacting}. On the other hand, for
more general weights in the six--vertex model we are not aware of such a direct connection.

\begin{proof}[Proof of Theorem \ref{theorem_particle_distribution}]
We start by noting that both parts of \eqref{eq_Transfer_marginal} are analytic in $b_1$ and
$b_2$. Thus, it suffices to prove it for small real $b_1$, $b_2$, which we will. Such small $b_1$
and $b_2$ guarantee that the pole of $f$ in \eqref{eq_f_definition} occurs at $\mathfrak s$ with
$|\mathfrak s|=s$ sufficiently small so as to ensure the existence of contour $C_{r'}$ in Lemma
\ref{lemma_deformation}.

The probability $P_\y(x_m=x;t)$ is the sum $\T^{(N)}_t(\y\to\x;1,1,b_1,b_2,1-b_1,1-b_2)$ over
$\x=(x_1<\dots<x_N)$ such that $x_m=x$. Thus, it is a sum over $-\infty<x_1<\dots<x_{m-1}<x$ and
over $x<x_{m+1}<\dots<x_N$. We will do these two summations in the formula for $\T^{(N)}$ of
Corollary \ref{corollary_stochastic_transfer}. Note that for the former summation we need the
contours to be large circles $C_R$, $R>1$, while for the latter we need the contours to be small
circles $C_r$, $r<1$.

The formula of Corollary \ref{corollary_stochastic_transfer} for $\T^{(N)}_t$ includes the
summation over $\sigma\in \mathfrak S_N$. Set $S_-(\sigma)=\sigma(\{1,\dots,m-1\})$ and
$S_+(\sigma)=\sigma(\{m,\dots,N\})$. Then $(S_-(\sigma),S_+(\sigma))$ is a partition of
$\{1,\dots,N\}$ into two sets such that $|S_-(\sigma)|=m-1$ and $|S_+(\sigma)|=N-m+1$. We identify
a permutation $\sigma$ with the four-tuple $(S_-(\sigma),S_+(\sigma), \sigma_-,\sigma_+)$, where
$\sigma_{-}\in\mathfrak S(m-1)$ and $\sigma_+\in\mathfrak S(N-m+1)$ are the permutations obtained
by restricting $\sigma$ onto $\{1,\dots,m-1\}$ and $\{m,\dots,N\}$, respectively (we use the
unique monotonous identifications of the elements of $S_-$ and $S_+$ with $\{1,\dots,m-1\}$ and
$\{1,\dots,N-m+1\}$, respectively). Note that the sign ($-1$ raised to the power the number of
inversions) of $\sigma$ can be written as:
$$
 \sgn(\sigma)=(-1)^{\kappa(S_-,S_+)} \sgn(\sigma_-) \sgn(\sigma_+),\quad
 \kappa(S_-,S_+):=\#\{(i,j)\mid i\in S_-,\, j\in S_+,\, i\ge j\}.
$$

We sum over $\sigma$ and over $\x$ in two steps. First, we fix a partition $(S_-,S_+)$ and sum over
all $\sigma_-$, $\sigma_+$. On the second step we sum over the partitions $(S_-,S_+)$. Let us
proceed with the first step.

We start by using Lemma \ref{lemma_deformation} to deform all the integration contours to circles
of radius $r'<1$. Then we can sum $\prod_{i\ge m} z^{x_i}_{\sigma(i)}$ over $x_m=x$,
$x<x_{m+1}<\dots <x_N<\infty$ using geometric series
\begin{multline*}
 \sum_{x<x_{m+1}<\dots<x_N} \prod_{i=m}^N z^{x_i}_{\sigma(i)} \\=  \frac{z_{\sigma(m+1)} z^2_{\sigma(m+2)}\cdots z^{N-m-1}_{\sigma(N)}}
 {(1-z_{\sigma(m+1)} z_{\sigma(m+2)}\cdots z_{\sigma(N)})(1-z_{\sigma(m+2)}\cdots z_{\sigma(N)}) \cdots
 (1-z_{\sigma(N)})}  \prod_{i=m}^N z_{\sigma(i)}^x.
\end{multline*}

After that we do the summation over $\sigma_+$ using Lemma \ref{lemma_TW_identity} (with $N$
replaced by $N-m+1$ and $\alpha=\tau^{-1}$). We arrive at
\begin{multline}
\label{eq_distr_after_step_1} P_\y(x_m=x;t)=\sum_{S_{-},S_{+},\sigma_-} \,
\sum_{x_1<\dots<x_{m-1}<x} (-1)^{\kappa(S_-,S_+)} (-1)^{\sigma_-} \frac{1}{(2\pi\ii)^N}
\oint\dots\oint \left(1-\prod_{i\in S_+} z_i\right)
\\ \times
 \prod_{\substack{i<j\\
i,j\in S_+}} (z_j-z_i) \prod_{i\in S_+} \frac{z_i^x}{1-z_i} \cdot
 \frac{\prod\limits_{1\le i<j<m} \Cr(z_{\sigma(i)}, z_{\sigma(j)}) \prod\limits_{1\le i<m\le j \le N }
\Cr(z_{\sigma(i)}, z_{\sigma(j)})}{\prod\limits_{1\le i<j\le N} \Cr(z_i,z_j)} \\
\times\prod\limits_{i=1}^{m-1} z_{\sigma(i)}^{x_i} \, \prod_{i=1}^N z_i^{-y_i-1}\left(
\frac{b_1+(1-b_1-b_2)z_i^{-1}}{1-b_2z_i^{-1}}\right)^t  dz_i,
\end{multline}
where
$$
 \Cr(u,v)=1-(1+\tau^{-1})u+\tau^{-1} uv,
$$
and the integration goes over the contours $C_{r'}$ with $r'<1$.

\smallskip
Next we deform all the integration contours in \eqref{eq_distr_after_step_1} to circles $C_R$ of a
very large radius $R$. We claim that the only new residues which appear in the contour deformation
are those at poles $z_i=1$, $i\in S_+$. In order to see that, note that according to Lemma
\ref{Lemma_Hall-littlewood} we can restore the summation over $\sigma_+$ in
\eqref{eq_distr_after_step_1} as follows:
$$
 \prod_{\substack{i<j\\
i,j\in S_+}} (z_j-z_i)
=\frac{(1-\tau^{-1})^{N-m+1}}{(\tau^{-1};\tau^{-1})_{N-m+1}}\sum_{\sigma_+\in \mathfrak S(N-m+1)}
(-1)^{\sigma_+} \prod_{m\le i<j\le N} \Cr(z_{\sigma(i)},z_{\sigma(j)}).
$$
Hence, replacing $\sum_{S_{-},S_{+},\sigma_-}\sum_{\sigma_+}$ in \eqref{eq_distr_after_step_1} by
$\sum_{\sigma\in S(N)}$, we see that the integrals almost fall into the framework of Lemma
\ref{lemma_deformation}, with the only difference being the presence of extra singularities at
points $z_i=1$, $i\in S_+$. Since Lemma \ref{lemma_deformation} claims that all other residues
cancel out in the contour deformation, we conclude that the only new residues which appear in the
contour deformation are those at poles $z_i=1$, $i\in S_+$.

Observe that according to the argument of Lemma \ref{lemma_deformation} the residue of
\eqref{eq_distr_after_step_1} at a multidimensional pole $\{z_i=1, i\in \widetilde S_+\}$,
$\widetilde S_+\subset S_+$ should be taken sequentially in the decreasing order over the indices
$i\in \widetilde S_+$. In other words, we first take the residue at $z_i=1$ with $i$ being the
largest element of $\widetilde S_+$, then continue with the second largest element, etc. Such
residue is the expression of the same type as summand in \eqref{eq_distr_after_step_1}, but with
smaller $N$, with $S_+$ replaced by $S_+\setminus \widetilde S_+$, and with an additional
prefactor $(-\tau)^{\kappa(\{1,\dots,N\}\setminus \widetilde S_+ ,\widetilde
S_+)}\cdot(-1)^{\kappa(S'_+,\tilde S_+)}$ . This is because $\Cr(z,1)=1-z$,
$\Cr(1,z)=-\tau^{-1}(1-z)$, $(z_i-z_j)\bigr|_{z_j=1}=-(1-z_i)$, $(z_i-z_j)\bigr|_{z_i=1}=1-z_j$. 

Let  $S'_+=S_+\setminus \widetilde S_+$ and set $k=|S_- \cup S'_+|$. With this notation we rewrite
\begin{equation}
 \kappa(\{1,\dots,N\}\setminus\widetilde S_+,\widetilde S_+)= \kappa(S_- \cup S'_+,\widetilde S_+)
=\kappa(S_- \cup S'_+,\mathbb Z_{>0})-k(k+1)/2.
\end{equation}

After we deformed the contours in \eqref{eq_distr_after_step_1} to large circles $C_R$, we sum
over $x_1<x_2<\dots<x_{m-1}<x$ using the geometric series
\begin{multline*}
 \sum_{x_1<x_2<\dots<x_{m-1}<x} \prod_{i=1}^{m-1} z_{\sigma(i)}^{x_i}
 \\=  \frac{z_{\sigma(m-1)}^{-1} z^{-2}_{\sigma(m-2)}\cdots z^{-m+1}_{\sigma(1)}}
 {(1-z^{-1}_{\sigma(m-1)} z^{-1}_{\sigma(m-2)}\cdots z^{-1}_{\sigma(1)})(1-z^{-1}_{\sigma(m-2)}\cdots z^{-1}_{\sigma(1)}) \cdots
 (1-z^{-1}_{\sigma(1)})}  \prod_{i=1}^{m-1} z_{\sigma(i)}^x.
\end{multline*}
After that we sum over $\sigma_-$ again using Lemma \ref{lemma_TW_identity}, but this time in
inversed variables renamed in the opposite order (i.e.\ $w_{N+1-i}=z_i^{-1}$), which is the
identity (we will use it with $N$ replaced by $m-1$ and with $\alpha=\tau$)
\begin{multline}
\label{eq_x14}
 Sym_{N}\left(\prod_{1\le i<j \le N} \frac{1-(1+\alpha^{-1}) w_i+\alpha^{-1}w_i w_j}{w_j-w_i} \prod_{i=1}^{N} \frac{w_i^{i-N}}{
 1- w_{N+1-i}^{-1} w_{N-i}^{-1} \cdots w^{-1}_{1}} \right)\\= (-1)^N\alpha^{-N(N-1)/2}\prod_{i=1}^{N} \frac{w_i}{1-w_i}.
\end{multline}

 The result is (recall that $S'_+=S_+\setminus \widetilde S_+$):
\begin{multline}
\label{eq_distr_after_step_2}
 P_\y(x_m=x;t)=\sum_{S_{-},S_{+}} \sum_{S'_+\subset S_+} (-1)^{m-1}
\tau^{\kappa(S_-\cup S'_+,\mathbb{Z}_{>0})-(m-1)(m-2)/2-k(k+1)/2}
\\
\times\frac{1}{(2\pi\ii)^{|S_-\cup S'_+|}} \oint\dots\oint \prod_{\substack{i\in S_-\\ j\in S'_+}
}
\frac{\Cr(z_i,z_j)}{z_j-z_i} \left(1-\prod_{j\in S'_+} z_j\right) \\ \times \prod_{\substack{i<j\\
i,j \in S_-\cup S'_+} } \frac{z_j-z_i}{\Cr(z_i,z_j)} \prod_{i\in S_-\cup S'_+} z_i^{x-y_i-1}
\left( \frac{b_1+(1-b_1-b_2)z_i^{-1}}{1-b_2z_i^{-1}}\right)^t \frac{ dz_i}{1-z_i},
\end{multline}
where the integration goes over the contours $C_R$. Note that $\kappa(S_- \cup S'_+,\mathbb
Z_{>0})$ is the sum of elements of $S_-\cup S'_+$. Let us clarify how $(-1)^{m-1}$ in
\eqref{eq_distr_after_step_2} arose. We had the sign $(-1)^{\kappa(S_-,S_+)}$ remaining from
\eqref{eq_distr_after_step_1}, the sign $(-1)^{\kappa(S_-\cup S'_+,\widetilde S_+)+\kappa(S'_+\cup
\widetilde S_+)}$ from taking residues at $\{z_i=1,i\in \widetilde S_+\}$, the sign $(-1)^{m-1}$
from using \eqref{eq_x14} and the sign $(-1)^{\kappa(S_-,S'_+)}$ absorbed into two products
involving $(z_j-z_i)$ in \eqref{eq_distr_after_step_2}. Since
$$
\kappa(S_-,S_+)+ \kappa(S_-\cup S'_+,\widetilde S_+)+\kappa(S'_+\cup \widetilde
S_+)+\kappa(S_-,S'_+)
$$
is even, the total sign is $(-1)^{m-1}$.

 It remains to make a partial summation in \eqref{eq_distr_after_step_2}. For that set
$S=S_-\cup S'_+$, which gives $|S|=k$. And for each fixed $S$ sum over all possible $S_-$.
Essentially, we are summing the second line in \eqref{eq_distr_after_step_2}. This is precisely
Lemma \ref{lemma_TW_identity_2} with $\alpha=\tau^{-1}$, $n=k$ (note that the cardinality of the
set is $m$ in Lemma \ref{lemma_TW_identity_2}, while $|S_-|=m-1$) and we arrive at the desired
formula.
\end{proof}

\subsection{Exponential moments of current}
\label{Section_flux} Up to this point we have followed a generalized version of the approach of
\cite{TW_determ}, but now we want to switch the gears and connect to the approach of
\cite{BigMac}, \cite{BCS}. The starting points of \cite{BigMac} and \cite{BCS} were the usage of
Macdonald difference operators and \emph{duality} for Markov chains, respectively. We do not know
how to generalize either of these techniques to the six--vertex model, but to reach a similar
outcome we can instead perform the summations in the result of Theorem
\ref{theorem_particle_distribution}.

\smallskip
Define the functions $N_x$ and $\eta_x$ of $\x=(x_1<x_2<\dots)\in \W_N$ ($N$ can be either a
positive integer of $+\infty$) via
$$
N_x(\x)=\#\{i:x_i\le x\},\quad \quad  \eta_x(\x)=\begin{cases} 1,& x_i= x\text{ for some
}i,\\0,&\text{otherwise.}\end{cases}
$$

For a function $f:\W_N\to\mathbb R$, let $\mathbb E_\y(f;t)$ represent the time $t$ expectation of
$f$, i.e.
 $$
  \mathbb E_\y(f;t)=\sum_{\x\in \W_N} \T^{(N)}_t(\y\to\x; 1,1,b_1,b_2,1-b_1,1-b_2) f(\x),
 $$
where $\T^{(N)}_t$ is the $t$--th power of the $N$--particle transfer matrix (we again allow $N$
to be $+\infty$ here).

\begin{proposition} Fix $N>0$, stochastic parameters of the six--vertex model $b_1$, $b_2$, and a positive
integer $t$; set $\tau=b_2/b_1$. Then for $L=0,1,2,\dots$ and any integer $x$ we have
\begin{multline}
\label{eq_observable_any_initial}
 \mathbb E_\y \left(\tau^{L N_x} \eta_x;t\right)=
 \sum_{k=1}^{\min(L,N)}\left( \prod_{i=0}^{k-2}\left(1-\tau^{1-k+L} \cdot \tau^i\right)\right)\cdot \sum_{|S|=k}
\tau^{\kappa(S,\mathbb Z_{>0})-k(k-1)/2+L-k} \\
\times \frac{1}{(2\pi\ii)^k} \oint \dots \oint \prod_{i,j\in S, \, i<j}
\frac{z_j-z_i}{1-z_i(1+\tau^{-1})+\tau^{-1}z_i z_j} \\ \times  \frac{1-\prod_{i\in S}
z_i}{\prod_{i\in S} (1-z_i)} \prod_{i\in S}
z_i^{x-y_i-1}\left(\frac{b_1+(1-b_1-b_2)z_i^{-1}}{1-b_2 z_i^{-1}}\right)^t dz_i,
\end{multline}
where the summation goes over $S\subset\{1,2,\dots,N\}$ of size $k$, $\kappa(S,\mathbb Z_{>0})$ is
the sum of elements in $S$, and contours are positively oriented large circles of equal radius
which contain all singularities of the integrand. When $k=1$, we agree that product
$\prod_{i=0}^{k-2}$ in \eqref{eq_observable_any_initial} is $1$.
\end{proposition}
\begin{proof}
We start from the result of Theorem \ref{theorem_particle_distribution}. Take $L\in \mathbb N$,
multiply \eqref{eq_Transfer_marginal} by $\tau^{mL}$ and sum over all $m=1,\dots,N$. The left-hand
side of \eqref{eq_Transfer_marginal} turns into
$$
 \mathbb E_\y \left(\tau^{L N_x} \eta_x;t\right).
$$
 The $m$--dependent part in the right-hand side of \eqref{eq_Transfer_marginal} is summed using
$q$--binomial theorem \eqref{eq_q_binomial} as follows:
\begin{multline}
\label{eq_qBinom_summation}
 \sum_{m=1}^k (-1)^{m-1} \tau^{m(m-1)/2-mk+mL} { {k-1} \choose
{m-1}}_{\tau}\\= \tau^{-k+L} \sum_{m=1}^k (-1)^{m-1} \tau^{(m-1)(m-2)/2}
\left(\tau^{1-k+L}\right)^{m-1} { {k-1} \choose
{m-1}}_{\tau}\\=\tau^{L-k}\prod_{i=0}^{k-2}\left(1-\tau^{1-k+L} \cdot \tau^i\right)
\end{multline}
Note that if $k>L$, then \eqref{eq_qBinom_summation} vanishes  and we arrive at the desired
formula.
\end{proof}

As a corollary, we obtain the following:

\begin{proposition} Fix stochastic parameters of the six--vertex model $b_1$, $b_2$, and a positive
integer $t$; set $\tau=b_2/b_1$ and let $step\in \W_\infty$ mean the step initial condition
$(1,2,3,\dots)$. Then for $L=0,1,2,\dots$ and any integer $x$ we have
\begin{multline}
\label{eq_observable_step_sum}
 \mathbb E_{step} \left(\tau^{L N_x};t\right)=1+
 \left(\tau^{-L}-1\right)\sum_{k=1}^{L}\left( \prod_{i=0}^{k-2}\left(1-\tau^{1-k+L} \cdot \tau^i\right)\right)
 \frac{\tau^{-k(k-1)/2+L-k}}{(\tau^{-1};\tau^{-1})_k}\\ \times \frac{1}{(2\pi\ii)^k}
  \oint \dots \oint \prod_{1\le i<j\le k}
\frac{z_j-z_i}{1-z_i(\tau^{-1}+1)+\tau^{-1}z_iz_j}\\\times \prod_{i=1}^k \frac{z_i^{x}
(\tau^{-1}-1)}{(1-z_i/\tau)(1-z_i)} \left(\frac{b_1+(1-b_1-b_2)z_i^{-1}}{1-b_2z_i^{-1}}\right)^t
dz_i,
\end{multline}
with integration over large positively oriented circles of equal radius containing all the
singularities of the integrand.
\end{proposition}
\begin{proof}
 We start from \eqref{eq_observable_any_initial}. Lemma \ref{lemma_T_infty} yields that the
 desired expectation $\mathbb E_{step} \left(\tau^{L N_x};t\right)$ is given by the $N\to\infty$
 limit of \eqref{eq_observable_any_initial}. Sending $N\to\infty$ in the right--hand side of
 \eqref{eq_observable_any_initial} is straightforward. Further, for the step initial condition $y_i=i$ we need to
  sum over the sets $S\subset\{1,2,\dots\}$.  Basically, after identification of
variables, we are computing the sum
$$
\sum_{1\le S_1<S_2<\dots<S_k} \tau^{S_1+S_2+\dots+S_k} \prod_{i=1}^k
z_i^{-S_i}=\prod_{i=1}^{k}\frac{(\tau/z_i)\cdots(\tau/z_k)}{1-(\tau/z_i)\cdots(\tau/z_k)}=
\prod_{i=1}^{k}\frac{1}{(z_i/\tau)\cdots(z_k/\tau)-1}
$$
(for convergence purposes we need the contours to be of radius larger than $\tau$ here).

The next step is to note that since the integration in our formulas goes over the same contours,
the results depend only on the symmetrization of the integrand.

Thus, using Lemma \ref{lemma_another_symmetrization} (with $\alpha=\tau^{-1}$) we conclude that
\begin{multline}
\label{eq_x1}
 \mathbb E_{step} \left(\tau^{L N_x} \eta_x;t \right)=
 \sum_{k=1}^{L}\left( \prod_{i=0}^{k-2}\left(1-\tau^{1-k+L} \cdot \tau^i\right)\right)
 \frac{\tau^{-k(k-1)/2+L-k}}{(\tau^{-1};\tau^{-1})_k} \\ \times  \frac{1}{(2\pi\ii)^k}
  \oint \dots \oint \prod_{1\le i<j\le k} \frac{z_j-z_i}{1-z_i(\tau^{-1}+1)+\tau^{-1}z_iz_j} \\
\times \left(1-\prod_{i=1}^k z_i\right) \prod_{i=1}^k
\frac{z_i^{x-1}(\tau^{-1}-1)}{(1-z_i/\tau)(1-z_i)}
\left(\frac{b_1+(1-b_1-b_2)z_i^{-1}}{1-b_2z_i^{-1}}\right)^t dz_i,
\end{multline}
with integration over large circles of equal radius.

Next, note that
\begin{multline*}
 \tau^{L N_x}=\tau^{L N_{x-1}}+ \tau^{L N_x} \eta_x (1-\tau^{-L})=\dots=(1-\tau^{-L})
 \sum_{a=-\infty}^x \tau^{L N_a} \eta_a+\lim_{a\to-\infty} \tau^{L N_a}\\=
 (1-\tau^{-L})
 \sum_{a=-\infty}^x \tau^{L N_a} \eta_a+1.
\end{multline*}
Thus, summing \eqref{eq_x1} we get the desired formula.
\end{proof}

The answer looks more pleasant after changing variables and moving to a modified (but
disconnected) set of contours.

\begin{figure}[h]
\begin{center}
 {\scalebox{1.2}{\includegraphics{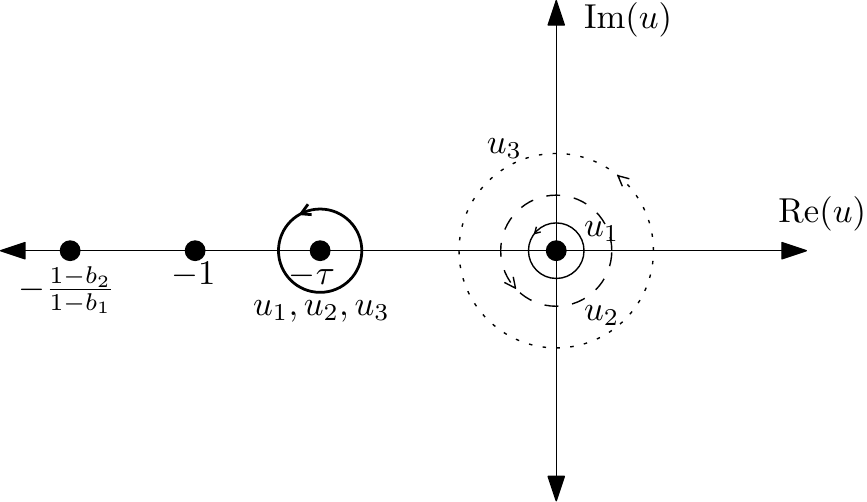}}}
 \caption{The contours for Theorem \ref{Theorem_observable_moment} with $N=3$. Each contour consists of 2 disjoint
 circles: one is a small loop around $-\tau$, another one is a small loop around $0$ (shown in
 solid, dashed, and dotted for $u_1$, $u_2$, and $u_3$, respectively)}
 \label{Figure_contours_disjoint}
\end{center}
\end{figure}

\begin{theorem} \label{Theorem_observable_moment}
Fix stochastic parameters of the six--vertex model $b_1$, $b_2$, and a positive integer $t$;
suppose $\tau=b_2/b_1<1$ and let
  $step\in \W_\infty$ mean the step initial condition $(1,2,\dots)$.
Then for $L=0,1,2,\dots$ and any integer $x$ we have
\begin{equation}
\label{eq_observable_step}
 \mathbb E_{step}\left(\tau^{L N_x};t\right)=\frac{\tau^{L(L-1)/2}}{(2\pi\ii)^L} \oint\dots\oint
 \prod_{A<B}\frac{u_A-u_B}{u_A-\tau u_B} \prod_{i=1}^L
 \left( \frac{1 +u_i\tau^{-1} \frac{1-b_1}{1-b_2}}{1
+u_i\frac{1-b_1}{1-b_2}} \right)^t \left(\frac{1+u_i}{1+u_i\tau^{-1}}\right)^x
 \frac{du_i}{u_i},
\end{equation}
where the positively oriented integration contour for $u_A$ includes $0$, $-\tau$,  but does not
include $-\frac{1-b_2}{1-b_1}$ or $\{\tau u_B\}_{B>A}$. (An example of such contours is shown in
Figure \ref{Figure_contours_disjoint}.)
\end{theorem}
\noindent{\bf Remark 1.} This is the first place where the condition $\tau<1$ is used.

\noindent{\bf Remark 2.}
 The ASEP limiting version of Theorem \ref{Theorem_observable_moment} was  proved in
 \cite[Theorem 4.20]{BCS} by  different methods.

\begin{proof}[Proof of Theorem \ref{Theorem_observable_moment}] We start from \eqref{eq_observable_step_sum} and make the change of variables
\begin{equation}
\label{eq_change_of_variables}
 z_i=\frac{1+u_i}{1+\tau^{-1}u_i},
 \quad \quad
 u_i=-\frac{1-z_i}{1-\tau^{-1}z_i}.
\end{equation}
Then
$$
 \frac{u_A-u_B}{u_A-\tau
 u_B}
 =\frac{(z_A-z_B)\tau^{-1}}{1 - z_B(1+\tau^{-1})+ \tau^{-1} z_A z_B },
\quad \quad
 \frac{b_1+(1-b_1-b_2)z^{-1}}{1-b_2z^{-1}}
=\frac{1 +\frac{u b_1}{b_2} \cdot \frac{1-b_1}{1-b_2}}{1 +u\frac{1-b_1}{1-b_2}}, $$
$$
 \frac{du}{u}=
 \frac{(\tau^{-1}-1)dz}{(1-z)(1-\tau^{-1}z)},
$$
and the large contours positively oriented contours in $z$ variables transform into small
negatively (i.e.\ clockwise) oriented contours around $-\tau$ in $u$--variables.

Thus, \eqref{eq_observable_step_sum} transforms into
\begin{multline}
\label{eq_x2} 1+ (\tau^{-L}-1)\sum_{k=1}^{L}\left( \prod_{i=0}^{k-2}\left(1-\tau^{1-k+L} \cdot
\tau^i\right)\right)
 \frac{\tau^{L-k}}{(\tau^{-1};\tau^{-1})_k} \\ \times  \frac{1}{(2\pi\ii)^k}
  \oint \dots \oint \prod_{1\le i<j\le k}
\frac{u_j-u_i}{u_j-\tau u_i}\, \prod_{i=1}^k \left(\frac{1+u_i}{1+u_i/\tau}\right)^{x}
\left(\frac{1+\frac{u_i b_1}{b_2}\cdot \frac{1-b_1}{1-b_2}}{1+u_i \frac{1-b_1}{1-b_2}}\right)^t
\frac{du_i}{u_i},
\end{multline}
with integration along the small contours around the pole at $-\tau$. Note that the factor
$\prod_{i<j}\frac{u_j-u_i}{u_j-\tau u_i}$ can be changed into $\prod_{i<j}\frac{u_i-u_j}{u_j-\tau
u_j}$ by renaming the contours. The numeric prefactor in the sum in \eqref{eq_x2} can be
transformed as follows:
\begin{equation}
\label{eq_x13} (\tau^{-L}-1) \prod_{i=0}^{k-2}\left(1-\tau^{1-k+L} \cdot \tau^i\right)
 \frac{\tau^{L-k}}{(\tau^{-1};\tau^{-1})_k}=\tau^{k(k-1)/2} { L\choose k}_{\tau} (-1)^k.
\end{equation}
We absorb the $(-1)^k$ factor in \eqref{eq_x13} into the orientation by making the $u$-contours
positively oriented. Now it remains to deform all the contours in \eqref{eq_x2} so that they
include both $0$ and $-\tau$, as in the statement of the Theorem. Note that in this deformation we
encounter residues and we should keep track of them. This is done in \cite[Lemma 4.21]{BCS}, which
shows that \eqref{eq_x2} is equal to the right-hand side of \eqref{eq_observable_step}.
\end{proof}

\subsection{Fredholm detetminants}

The aim of this section is to show that a certain generating function of the observables $\mathbb
E_{step}\left(\tau^{L N_x};t\right)$ of Theorem \ref{Theorem_observable_moment} can be written as
a Fredholm determinant. The argument here is very similar to those of \cite{BigMac}, \cite{BCS}
and we omit some technical details. A more detailed exposition can be found in \cite[Section
3.2]{BigMac} and \cite[Section 3 and Section 5]{BCS}. Below we record some notations and
background on Fredholm  determinants.

\begin{definition} \label{Def_fredholm}
 Let $K(x,y)$ be a meromorphic function of two complex variables, that we will refer to as a kernel, and let
$\Gamma\subset \mathbb C$ be a curve. Suppose that $K$ has no singularities on $\Gamma\times
\Gamma$. Then the Fredholm determinant of the kernel $K$, notation
$\det\bigl(1+K\bigr)_{\Gamma}$, is defined as the sum of the series of complex integrals
\begin{equation}
\label{eq_Fredholm_series}
 \det\bigl(1+K\bigr)_{\Gamma}=1+\sum_{n=1}^{\infty} \frac{1}{n! (2\pi \ii)^n} \int_{\Gamma}\dots\int_{\Gamma}
\det\bigl(K(z_i,z_j)\bigr)_{i,j=1}^n dz_1\cdots dz_n.
\end{equation}
\end{definition}
\noindent{\bf Remark.}
 Note two differences with the usual definition of the Fredholm determinant for the kernel: the
complex integration and prefactor $\frac{1}{(2\pi\ii)^n}$.  Further, the Fredholm determinant of a
kernel is typically identified with the Fredholm determinant of the corresponding integral
operator. We are not going to use any operator theory and, thus, such an identification is not
important to us. In what follows we will merely stick to the definition via the series expansion
\eqref{eq_Fredholm_series}.

\smallskip

The following statement, known as Hadamard's inequality (see e.g.\ \cite[Problem 2.1.P23]{HJ}), is
useful in the analysis of the series \eqref{eq_Fredholm_series}.

\begin{lemma}[Hadamard's inequality] Let $A=\{A_{ij}\}$ be a complex $N\times N$ matrix. Then
$$
 |\det(A)|\le \prod_{i=1}^N \sqrt{\sum_{j=1}^N (A_{ij})^2}.
$$
\end{lemma}

\begin{corollary}
 If $\Gamma$ is a smooth curve of finite length, then the series in \eqref{eq_Fredholm_series} absolutely converges.
\end{corollary}
\begin{proof} Since $K$ has no singularities on $\Gamma\times \Gamma$, there is a constant $B$
such that $|K(x,y)|\le B$ for all $(x,y)\in\Gamma\times \Gamma$. Then using  Hadamard's
inequality, the absolute value of the $n$th term in \eqref{eq_Fredholm_series} is bounded from
above by
$$
 \frac{1}{n!} \left(\frac{BS}{2\pi}\right)^n  n^{n/2},
$$
where $S$ is the length of $\Gamma$. Since $n!>(n/e)^n$, the convergence readily follows.
\end{proof}

Now we are ready to state the main result of this section, in which we adopt the notation of
Section \ref{Section_flux}.

\begin{theorem} \label{Theorem_observable_fredholm} Fix stochastic parameters of the six--vertex model $b_1$, $b_2$, and
a positive integer $t$; suppose $\tau=b_2/b_1<1$ and let $step\in \W_\infty$ mean the initial
condition $(1,2,3,\dots)$. Then for all $\zeta\in\mathbb C\setminus\mathbb R_{\ge 0}$  we have
\begin{equation}
\label{eq_observable_Fredholm}
 \mathbb E_{step}\left(\frac{1}{(\zeta\tau^{
 N_x};\tau)_\infty};t\right)=\det\left(I+K^{b_1,b_2}_\zeta\right)_{C_r},
\end{equation}
where $C_r$ is the positively oriented circle with center at the origin and of radius $r$
satisfying
$$
 \tau<r<\frac{1-b_2}{1-b_1} \tau ,
$$
and the kernel $K^{b_1,b_2}_\zeta$ is given by
\begin{equation}
\label{eq_kernel}
 K^{b_1,b_2}_\zeta (w,w')=\frac{1}{2\ii} \int_{1/2+\ii\mathbb R} \frac{(-\zeta)^s}{\sin(\pi s)}
 \cdot \frac{g(w; b_1,b_2, x, t)}{g(\tau^s w; b_1,b_2, x, t)} \cdot \frac{ ds}{\tau^s w-w'},
\end{equation}
where the integration contour is oriented from bottom to top and
$$
 g(z; b_1,b_2, x, t)=\left( 1 + z \tau^{-1} \frac{1-b_1}{1-b_2} \right)^t
 \left(\frac{1}{1+z\tau^{-1}}\right)^x.
$$
\end{theorem}

The proof is based on two lemmas, for which we first need to introduce additional notations.

A partition $\lambda\vdash k$ of an integer $k>0$ is an ordered sequence of integers (``parts'')
$\lambda_1\ge\lambda_2\ge\dots\ge0$ such that $\sum_{i=1}^{\infty}\lambda_i=k$. The length
$\ell(\lambda)$ is the number of non-zero parts in $\lambda$ and the number $k$ is denoted
$|\lambda|$. An alternative encoding of $\lambda$ is $\lambda=1^{m_1} 2^{m_2}\dots$, which means
that $\lambda$ has $m_1$ parts equal to $1$, $m_2$ parts equal to $2$, etc. In particular, this
implies $|\lambda|=\sum_{i\ge 1} i m_i$.

\begin{lemma}\label{Lemma_integral_into_sum}
Fix $\alpha\in \C\setminus \{0\}$ and $0<\tau<1$. Consider a meromorphic function $f(z)$ which has
a pole at $\alpha$ but does not have any other poles in an open neighborhood of the line segment
connecting $\alpha$ to 0. For such a function and for any $k\geq 1$, define
\begin{equation}\label{eq_muk_1}
\mu_k= \frac{\tau^{\frac{k(k-1)}{2}}}{(2\pi \ii)^k} \int \cdots \int \prod_{1\leq A<B\leq k}
\frac{z_A-z_B}{z_A-\tau z_B} \prod_{i=1}^{k} f(z_i) \frac{dz_i}{z_i}
\end{equation}
where the integration contour for $z_A$ contains 0, $\alpha$ but does not include any other poles
of $f$ or $\{\tau z_B\}_{B>A}$. (For instance, this is the case in \eqref{eq_observable_step}.)
Then
\begin{multline}
\label{eq_muk_2}
\mu_k= (\tau;\tau)_k\sum_{\substack{\lambda\vdash k\\
\lambda=1^{m_1}2^{m_{2}}\cdots}} \frac{1}{m_1!m_2!\cdots} \, \frac{1}{(2\pi \ii)^{\ell(\lambda)}}
\\ \times \int_C \cdots \int_C \det\left[\frac{-1}{w_i
\tau^{\lambda_i}-w_j}\right]_{i,j=1}^{\ell(\lambda)} \prod_{j=1}^{\ell(\lambda)}  f(w_j)f(\tau
w_j)\cdots f(\tau^{\lambda_j-1}w_j) dw_j,
\end{multline}
where the integration contour $C$ for $w_j$ is a closed curve which contains 0, $\alpha$ and no
other poles of $f$; and its image under multiplication by any positive power of $\tau$ lies inside
$C$. (For instance, for $f(z)=g(z)/g(\tau z)$ with function $g$ of Theorem
\ref{Theorem_observable_fredholm}, and $\alpha=-\tau$, this is the contour $C_r$ of the same
theorem.)
\end{lemma}
\begin{proof}
The proof is via residue calculus and it is given in \cite[Proposition 5.2]{BCS}.
\end{proof}

\begin{lemma} \label{Lemma_sum_into_fredholm}
 Take $0<\tau<1$ and let $f$ and $g$ be two meromorphic functions such that $f(z)=g(z)/g(\tau z)$.
 Suppose that $f$ satisfies the assumptions of Lemma \ref{Lemma_integral_into_sum}, contour $C$ is as in that lemma
  and $\mu_k$ is given by \eqref{eq_muk_2}. Further, assume that for a certain $\delta$,  $0<\delta<1$, the expression
$$
\left| \frac{g(w)}{g(\tau^s w)} \cdot \frac{1}{\tau^s w- w'} \right|
$$
is uniformly bounded over $\Re(s)\ge \delta$, $w\in C$, $w'\in C$. Then for all sufficiently small
complex numbers $\zeta\in\mathbb C\setminus \mathbb R_{\ge 0}$ so that the series below converges,
we have
 \begin{equation}
 \label{eq_series_into_fredholm}
  \sum_{k=0}^{\infty} \frac{\mu_k \zeta^k}{(\tau;\tau)_k}=\det\bigl(1+K\bigr)_{C},
 \end{equation}
 where
 \begin{equation}
\label{eq_kernel_def}
  K(w,w')=\frac{1}{2\ii} \int_{\delta+\ii\mathbb R} \frac{(-\zeta)^s}{\sin(\pi s)} \cdot
 \frac{g(w; b_1,b_2, x, t)}{g(\tau^s w; b_1,b_2, x, t)} \cdot \frac{ ds}{\tau^s w-w'},
 \end{equation}
 the integration contour is oriented from bottom to top, and
$(-\zeta)^s$ is understood as $\exp\bigl(s\ln(-\zeta)\bigr)$ with principal branch of the
logarithm with cut along negative real semi-axis (corresponding to positive $\zeta$).
\end{lemma}
\begin{proof}
 We present here a sketch of the proof, a detailed exposition can be found in \cite[Section 3.2]{BigMac} and
\cite[Section 3]{BCS}. Plugging the definition of $\mu_k$ \eqref{eq_muk_2} into the sum $ \sum_{k\ge 0} \frac{\mu_k
\zeta^k}{(\tau;\tau)_k}$ we get
\begin{multline}
\label{eq_x3} \sum_{ \lambda=1^{m_1}2^{m_{2}}\cdots} \frac{\zeta^{|\lambda|}}{m_1!m_2!\cdots} \,
\frac{1}{(2\pi \ii)^{\ell(\lambda)}}
\\ \times \int_C \cdots \int_C \det\left[\frac{-1}{w_i
\tau^{\lambda_i}-w_j}\right]_{i,j=1}^{\ell(\lambda)} \prod_{j=1}^{\ell(\lambda)}  f(w_j)f(\tau
w_j)\cdots f(\tau^{\lambda_j-1}w_j) dw_j,
\end{multline}
with summation going over all partitions $\lambda$. Take $N=0,1,2,\dots$ and sum \eqref{eq_x3}
first over all $\lambda$ such that $\ell(\lambda)=N$. Now $\lambda$ is a sequence of integers
$\lambda_1\ge\lambda_2\ge\dots\ge\lambda_N\ge 1$. Let us remove the ordering assumption and
instead sum over $\lambda_1\ge 1$, \dots, $\lambda_N\ge 1$. This turns \eqref{eq_x3} into
\begin{multline}
 \sum_{N=0}^\infty \frac{1}{N!(2\pi\ii)^N} \sum_{\lambda_1=1}^\infty \cdots
\sum_{\lambda_N=1}^\infty  \zeta^{\lambda_1+\dots+\lambda_N}
\\ \times \int_C \cdots \int_C \det\left[\frac{-1}{w_i
\tau^{\lambda_i}-w_j}\right]_{i,j=1}^{N} \prod_{j=1}^{N}  f(w_j)f(\tau
w_j)\cdots f(\tau^{\lambda_j-1}w_j) dw_j.
\end{multline}
Using the definition of $g(z)$, we get
\begin{equation}
\label{eq_x4}
\sum_{N=0}^\infty \frac{1}{N!(2\pi\ii)^N} \sum_{\lambda_1=1}^\infty \cdots
\sum_{\lambda_N=1}^\infty  \zeta^{\lambda_1+\dots+\lambda_N}
 \int_C \cdots \int_C \det\left[\frac{-1}{w_i
\tau^{\lambda_i}-w_j}\right]_{i,j=1}^{N} \prod_{j=1}^{N}  \frac{g(w_j)}{g(\tau^{\lambda_j}w_j)} dw_j.
\end{equation}
Interchanging summation and integration and using the linearity of the determinant, we obtain
\begin{equation}
\label{eq_x5}
\sum_{N=0}^\infty \frac{1}{N!(2\pi\ii)^N}
 \int_C \cdots \int_C \det\left[\sum_{\lambda_i=1}^{\infty}\frac{-\zeta^{\lambda_i}}{w_i
\tau^{\lambda_i}-w_j} \frac{g(w_i)}{g(\tau^{\lambda_i}w_i)} \right]_{i,j=1}^{N} \prod_{j=1}^{N} dw_j.
\end{equation}
To finish the proof it remains to show that
\begin{equation}
\label{eq_x6}
 \sum_{a=1}^{\infty}\frac{-\zeta^a}{w_i
\tau^{a}-w_j} \frac{g(w_i)}{g(\tau^{a}w_i)}=K(w_i,w_j).
\end{equation}
For that we first note that the integrand in the definition \eqref{eq_kernel_def} of $K$ decays
exponentially fast as $|s|\to\infty$ along the vertical line $\Re(s)=\delta$ (due to the decay of
$1/\sin(\pi s)$.) Further, because of the same decay, the integral in the definition of $K$ can be
obtained as $k\to\infty$ limit of the same integral with contour $\Re(s)=\delta$ replaced by the
closed half-circle $C(k)$, consisting of the vertical line joining $\delta-k\ii$ with
$\delta+k\ii$ and right half of the circle of radius $k$ with center at $\delta$. The integral
over $C(k)$ can be computed as a sum of the residues in points $m=1,\dots,k$ using
$${\rm
Res}_{s=m} \frac{\pi}{\sin(\pi s)}=(-1)^{m}, \quad m=1,2,\dots$$ (note that an additional minus
sign arises because of the orientation of the vertical line $\delta+\ii\mathbb R$ in the theorem).
Sending $k\to\infty$ we arrive at \eqref{eq_x6}, see \cite[Proof of Theorem 3.2.11]{BigMac} for
more technical details.
\end{proof}

\begin{proof}[Proof of Theorem \ref{Theorem_observable_fredholm}]
We start by noting that \eqref{eq_observable_step} has the form of Lemma \ref{Lemma_integral_into_sum} with
$$
 f(z)=\left(\frac{1 +z\tau^{-1} \frac{1-b_1}{1-b_2}}{1
+z\frac{1-b_1}{1-b_2}} \right)^t \left(\frac{1+z}{1+z\tau^{-1}}\right)^x
$$
and $\alpha=-\tau$. Thus, setting
$$
 g(z; b_1,b_2, x, t)=\left( 1 + z \tau^{-1} \frac{1-b_1}{1-b_2} \right)^t \left(\frac{1}{1+z\tau^{-1}}\right)^x
$$
so that
$$
 f(z)=g(z)/g(\tau z)
$$
we can use Lemma \ref{Lemma_integral_into_sum} and then Lemma \ref{Lemma_sum_into_fredholm}.
Noting that by the $q$--binomial theorem \eqref{eq_q_binomial_3}
$$
 \sum_{k=0}^{\infty} \frac{\tau^{k N_x} \zeta^k}{(\tau;\tau)_k}=\frac{1}{(\zeta \tau^{k N_x(t)};\tau)_\infty},
$$
and that due to the bound $0<\tau^{N_x}<1$ we can interchange the order of summation and taking
the expectation,  we arrive at the statement of Theorem \ref{Theorem_observable_fredholm} for
small values of $\zeta$. By an analytic continuation argument this readily implies the statement
for all $\zeta\in \mathbb C\setminus \mathbb R_+$, cf.\ \cite[Proof of Theorem 3.2.11]{BigMac} for
a similar argument.
\end{proof}

\section{Asymptotics}

The main result of this section is summarized in the following theorem. We now switch back to
using the Markov chain $\x^{b_1,b_2}(t)$ of Section \ref{Section_P_as_interacting} in our
notations. Recall that due to Proposition \ref{Proposition_T_infty}, the fixed $t$ distribution of
$\x^{b_1,b_2}(t)$ is the result of the application of the $t$--th power of the transfer matrix
$\T^{(\infty)}$ to the step initial condition $(1,2,3,\dots)$.

\begin{theorem}\label{Theorem_asymptotics}
Let $0<b_2<b_1<1$ and set $\tau=\frac{b_2}{b_1}$, $\kappa:= \frac{1-b_1}{1-b_2}$. Let $N_x(t)$
denote the number of particles (non-strictly) to the left of the point $x$ in $\x^{b_1,b_2}(t)$ of
Section \ref{Section_P_as_interacting}. For any $\nu \in (\kappa,\kappa^{-1})$ and $s\in \R$ we
have
$$
\lim_{L\to \infty} \PP\left(\frac{m_\nu L -N_{\nu L}(L)}{\sigma_{\nu} L^{1/3}}\le h\right) =
F_{{\rm GUE}}(h),
$$
where
$$
m_{\nu} := \frac {\left(\sqrt{\nu}-\sqrt{\kappa}\right)^2}{{1-\kappa}},\qquad \sigma_{\nu} :=
\frac{\kappa^{1/2}\nu^{-1/6}}{1-\kappa} \left(\big(1-\sqrt{\nu\kappa}\big)\big(\sqrt{\nu/
\kappa}-1\big)\right)^{2/3},
$$

and $F_{{\rm GUE}}$ is the GUE Tracy-Widom distribution.
\end{theorem}
\noindent{\bf Remark 1.}
 There are some heuristic ways to understand the condition that $\nu\in
(\kappa,\kappa^{-1})$. If rather than starting with step initial data, one considers initial data
with a single particle started at the origin, then a quick calculation reveals that the law of
large numbers for the location of this particle after long time $t$ has velocity $\kappa$. If
instead of looking at particles, we consider holes (i.e.\ spots with no particles) and start with
 a single hole at the origin, then a similar calculation reveals that the law of large numbers for
the location of this hole after long time $t$ has velocity $\kappa^{-1}$.  Of course, this
reasoning neglects the effects of the other particles/holes but remarkably gives the correct
interval for analyzing fluctuation behavior.

There is a general KPZ--scaling theory (cf.\ \cite{Spohn_KPZ}), which should predict the interval
$(\kappa,\kappa^{-1})$, the centering $m_{\nu}$ and the scaling $\sigma_{\nu}$. We do not check
whether Theorem \ref{Theorem_asymptotics} conforms with such predictions. The flux function $j(y)$
one would need to make this check was computed in \cite{GwaSpohn} as equation (6).

\smallskip

\noindent{\bf Remark 2.}
 The asymptotics we now perform can be adapted to degenerations of the
process described in Section \ref{Section_P_as_interacting}.

\medskip

The proof of Theorem \ref{Theorem_asymptotics} is a steepest descent analysis of the integrals of
Theorem \ref{Theorem_observable_fredholm}. Proofs of similar style were performed previously in
\cite{BigMac}, \cite{BCF}, \cite{BCR},  \cite{F_qT}.

\smallskip


The following simple lemma shows how the observable considered in Theorem
\ref{Theorem_observable_fredholm} can be used to study the convergence of probability
distributions.

\begin{lemma}\label{Lemma_observable_to_probability}
Consider a sequence of functions $\{f_n\}_{n\geq 1}$ mapping $\R\to [0,1]$ such that for each $n$,
$f_n(x)$ is strictly decreasing in $x$ with a limit of $1$ at $x=-\infty$ and $0$ at $x=\infty$,
and for each $\delta>0$, on $\R\setminus[-\delta,\delta]$ $f_n$ converges uniformly to ${\bf
1}(x\leq 0)$. Define the $r$-shift of $f_n$ as $f^r_n(x) = f_n(x-r)$. Consider a sequence of
random variables $X_n$ such that for each $r\in \R$,
\begin{equation*}
\EE[f^r_n(X_n)] \to p(r)
\end{equation*}
and assume that $p(r)$ is a continuous probability distribution function. Then $X_n$ converges
weakly in distribution to a random variable $X$ which is distributed according to $\PP(X\leq r) =
p(r)$.
\end{lemma}
\begin{proof} See  \cite[Lemma 4.1.39]{BigMac} \end{proof}

The main part of the proof of Theorem \ref{Theorem_asymptotics} is  the following proposition.

\begin{proposition} \label{Proposition_analytical_converence}
In the notations of Theorem \ref{Theorem_asymptotics}, with $K_\zeta$ of Theorem
\ref{Theorem_observable_fredholm}, and $x=\lfloor \nu L \rfloor$  we have
$$
 \lim_{L\to\infty} \det\left(1+K^{b_1,b_2}_{\zeta(L)}\right)_{C_r}=F_{GUE}(h),\quad \text{ where } \zeta(L)= -\tau^{-m_\nu L +h\sigma_\nu
 L^{1/3}}.
$$
\end{proposition}
 This is proved in Sections \ref{section_proof_1} and \ref{section_proof_2}.

\begin{proof}[Proof of Theorem \ref{Theorem_asymptotics}]
We use Lemma \ref{Lemma_observable_to_probability} with functions
$$
 f_L(z)=\frac{1}{\left(-\tau^{-L^{1/3} z};\tau\right)_\infty},\quad L=1,2,\dots.
$$
 Clearly, $f_L(z)$ is is a monotonously decreasing function of real
argument $z$, $\lim_{z\to+\infty}f_L(z)=0$ and $\lim_{z\to-\infty} f_L(z)=1$. Further,
$f_L(z)=f_1(L^{1/3} z)$ which implies that for any $z<0$ we have $\lim_{L\to\infty} f_L(z)=1$ and
for any $z>0$ we have $\lim_{L\to\infty} f_n(z)=0$. We conclude that $f_L(z)$ satisfy the
assumptions of Lemma \ref{Lemma_observable_to_probability}.

We choose $X_L$ of Lemma \ref{Lemma_observable_to_probability} to be
$$
 X_L=\frac{m_\nu L -N_{\nu L}(L)}{L^{1/3} \sigma_\nu}.
$$
Now Lemma \ref{Lemma_observable_to_probability}, Proposition \ref{Proposition_T_infty} and Theorem
\ref{Theorem_observable_fredholm} reduce Theorem \ref{Theorem_asymptotics} to Proposition
\ref{Proposition_analytical_converence}.
\end{proof}

\subsection{Proof of Proposition \ref{Proposition_analytical_converence}}

\label{section_proof_1}

The proof of Proposition \ref{Proposition_analytical_converence} consists of two parts. One is a
formal steepest-descent computation of the leading asymptotic term. Second part gives estimates
proving that all other terms do not contribute. We start with the first part.

\bigskip

We have
\begin{equation}
\label{eq_kernel_plugged}
 K^{b_1,b_2}_{\zeta(L)} (w,w')=\frac{1}{2\ii} \int_{1/2+\ii\mathbb R} \frac{\tau^{s\left(-m_\nu L +h\sigma_\nu
 L^{1/3}\right)}}{\sin(\pi s)}
 \cdot \frac{g(w; b_1,b_2, \nu L , L)}{g(\tau^s w; b_1,b_2, \nu L , L)} \cdot \frac{ ds}{\tau^s w-w'},
\end{equation}
where the integration contour is oriented from bottom to top and
$$
 g(z; b_1,b_2, x, t)=\left( 1 + z \tau^{-1} \frac{1-b_1}{1-b_2} \right)^t
 \left(\frac{1}{1+z\tau^{-1}}\right)^x.
$$
Recall that the powers $\tau^u$ here should be understood as $\exp(u \ln(\tau) )$ and
set\footnote{In order to simplify the exposition we omit the integer part and set $x=\nu L$
instead of $x=\lfloor \nu L\rfloor$ in the following argument. The arising additional factor
$(1+z\tau^{-1})^{\nu L-\lfloor \nu L\rfloor}$ plays no role for the asymptotics.} $t=L$, $x=\nu L$
, we write and write \eqref{eq_kernel_plugged} as
\begin{equation}
\label{eq_kernel_plugged_2}
 K^{b_1,b_2}_{\zeta(L)} (w,w')=\frac{1}{2\ii} \int_{1/2+\ii\mathbb R} \frac{\exp\left[L(G(w)-G(\tau^s
 w)) + L^{1/3}h \sigma_\nu(\ln(\tau^s w)-\ln(w))\right]
 \cdot  ds}{(\tau^s w-w')\sin(\pi s)},
\end{equation}
where we choose the branch of the logarithm with the cut along negative real axis here and
$$
G(z) = \ln\big(1+ z \kappa \tau^{-1}\big) - \nu \ln\big(1+ z \tau^{-1}\big) + m_{\nu} \ln(z).
$$

The form of the integrand in \eqref{eq_kernel_plugged_2} suggests a change of variables $v=\tau^s
w$. However, we should be careful here as the map $s\mapsto \tau^s$ is periodic. More precisely,
as $s$ varies over the vertical line $1/2+\ii\mathbb R$, $\tau^s w$ wraps around a circle. In
order to make sense of that feature we subdivide the integration contour $1/2+\ii\mathbb R$ into
finite contours $\mathcal L_k$, $k\in \mathbb Z$:
$$
 \mathcal L_k=\{s\mid \Re(s)=1/2,\quad \pi \ln(\tau^{-1})(-1+2k) \le Im(s)< \pi \ln(\tau^{-1})(1+2k)
 \}.
$$
On each $\mathcal L_k$, the map $s\mapsto \tau^s w$ is a bijection (onto a circle) and we can
introduce the variable $v=\tau^s w$. We get
\begin{equation}
\label{eq_kernel_plugged_3}
 K^{b_1,b_2}_{\zeta(L)} (w,w')=\frac{1}{2\ii} \sum_{k=-\infty}^{\infty} \int_{C_{\sqrt{\tau}|w|}}
 \frac{\exp\left[L(G(w)-G(v)) + L^{1/3}h \sigma_\nu(\ln(v)-\ln(w))\right]
 \cdot  dv}{(v-w')\sin\left(\frac{\pi}{\ln(\tau)} (\ln(v/w)+2\pi\ii k)\right) \ln(\tau) v},
\end{equation}
where the integration contour $C_{\sqrt{\tau}|w|}$ is a (clockwise-oriented) circle of radius
$\sqrt{\tau}|w|$, which is $\sqrt{\tau}r$ taking into the account the definition of the
$w$--contour. One might be uneasy about various branches of the logarithms that we choose in
\eqref{eq_kernel_plugged_3}, but in the end in the relevant domains of integration all the
arguments of logarithms will be close to being real positive and principal branch will work. Note
that $\sin\left(\frac{\pi}{\ln(\tau)} (\ln(v/w)+2\pi\ii k)\right)$ grows exponentially in $|k|$,
as $k\to\infty$, therefore the series in \eqref{eq_kernel_plugged_3} converges exponentially fast.
Moreover, this property also justifies the termwise $L\to\infty$ limit in
\eqref{eq_kernel_plugged_3} which we will perform.

The crucial property which we will further use is that the integrand in
\eqref{eq_kernel_plugged_3} has a pole at $v=w$ when $k=0$ but not for other $k$.

Note that at this stage we have two contours in play: the contour $C_r$, where $w$--variables live
($\tau<r<\tau/\kappa$) and contour $C_{\sqrt{\tau}|w|}$, where $v$ lives. Since the Fredholm
determinant we deal with is defined as a sum of  complex integrals of meromorphic functions, and
as long as we avoid the singularities we can deform both $w$- and $v$-contours without changing
the value of integrals and, thus, of the determinant.

Our next aim is to deform both contours to new ones, where the $L\to\infty$ asymptotics can be
performed.

For that we need to understand how the real part of $G(z)$ behaves as $z$ varies over $\mathbb C$.
We start by taking derivative to find critical points:
$$
 G'(z)= \frac{\kappa \tau^{-1}}{1+ z \kappa \tau^{-1}} - \frac{\nu \tau^{-1}}{1+ z \tau^{-1}} +
 \frac{m_{\nu}}{z}.
$$
Plugging in the value of $m_\nu$
$$
m_{\nu}= \frac{1-b_1}{b_1-b_2}
\left(1-\sqrt{\nu/\kappa}\right)^2=\frac{\left(1-\sqrt{\nu/\kappa}\right)^2}{\kappa^{-1}-1},
$$
we get
$$
 G'(z)=
 \frac{\kappa(1-\sqrt{\nu\kappa})^2}{ \tau^2(1-\kappa)}  \frac{(z-\varrho)^2}{z(1+ z \kappa \tau^{-1})(1+ z \tau^{-1})},
$$
where
$$
\varrho = -\tau \frac{1-\sqrt{\nu/\kappa}}{1-\sqrt{\nu\kappa}}.
$$
It follows that $G(z)$ has a unique (double) critical point at $\varrho$. Further,
$$\frac{G'''(\varrho)}{2}=\frac{\kappa(1-\sqrt{\nu\kappa})^2}{ \tau^2(1-\kappa)}\cdot
 \frac{1}{\varrho(1+ \varrho \kappa \tau^{-1})(1+ \varrho \tau^{-1})}
=-\frac{\kappa^{3/2}(1-\sqrt{\nu\kappa})^5}{ \tau^3(1-\kappa)^3(1-\sqrt{\nu/\kappa})\sqrt{\nu}},
$$
and the last expression is precisely $\left(\sigma_\nu/\varrho\right)^3$. We conclude that Taylor
expansion of $G(z)$ near $\varrho$ is
\begin{equation}
\label{eq_Taylor}
 G(z)=G(\varrho)+ \left(\frac{\sigma_v}{\varrho}\right)^3 \cdot \frac{(z-\varrho)^3}{3} +
 o(z-\varrho)^3.
\end{equation}
 This decomposition implies that near
$\varrho$ there are $6$ branches of level lines $\Re(G(z))=\Re(G(\varrho))$ departing $\varrho$ at
$6$ directions with angles $2\pi/3$ between adjacent ones. A sketch of them is shown in Figure
\ref{Figure_Contours}. Let us explain why the picture looks as shown. Note that the desired level
lines are smooth curves which cannot intersect each other, since any point of intersection would
have to be a critical point of $G(z)$. Also due to the maximum principle for harmonic functions,
any closed loop formed by the level lines should enclose $0$, $-\tau$ or $-\tau/\kappa$, which are
the only points where $\Re(G(z))$ is not harmonic. Further we can trace the signs of
$\Re(G(z))-\Re(G(\varrho))$ along the real axis. Simple considerations imply that there exist
points $d_1$, $d_2$ and $d_3$ such that $d_1<-\tau/\kappa<d_2<-\tau<d_3<0$ and
$\Re(G(z))-\Re(G(\varrho))>0$ changes the signs at these points. Namely,
$\Re(G(z))-\Re(G(\varrho))>0$ on $(-\infty,d_1)$; $\Re(G(z))-\Re(G(\varrho))<0$ on $(d_1,d_2)$;
$\Re(G(z))-\Re(G(\varrho))>0$ on $(d_2,d_3)$; $\Re(G(z))-\Re(G(\varrho))<0$ on $(d_3,\varrho)$;
and $\Re(G(z))-\Re(G(\varrho))>0$ on $(\varrho,+\infty)$. In Figure \ref{Figure_Contours} the
points $d_1$, $d_2$, $d_3$ are (negative) intersections of dashed lines with the real axis.
Finally, we claim that for any large value $M\gg 1$, the equation $\Re(G(z))=\Re(G(\varrho))$ has
no solutions on the circle $|z|=M$. Indeed, for large $z$,
\begin{multline*}
\Re(G(z))= \ln\big|z \kappa \tau^{-1}(1+z^{-1} \kappa^{-1} \tau)\big| - \nu \ln\big|z
\tau^{-1}(1+z^{-1}\tau)\big| + m_{\nu} \ln|z|
\\
=\ln|z|(1-\nu+m_\nu)+\ln(\kappa)-\ln(\tau)+\nu \ln(\tau)+o(1),
\end{multline*}
and $1-\nu+m_\nu>0$ for all $\kappa<\nu<\kappa^{-1}$.

\begin{figure}[h]
\begin{center}
 {\scalebox{1}{\includegraphics{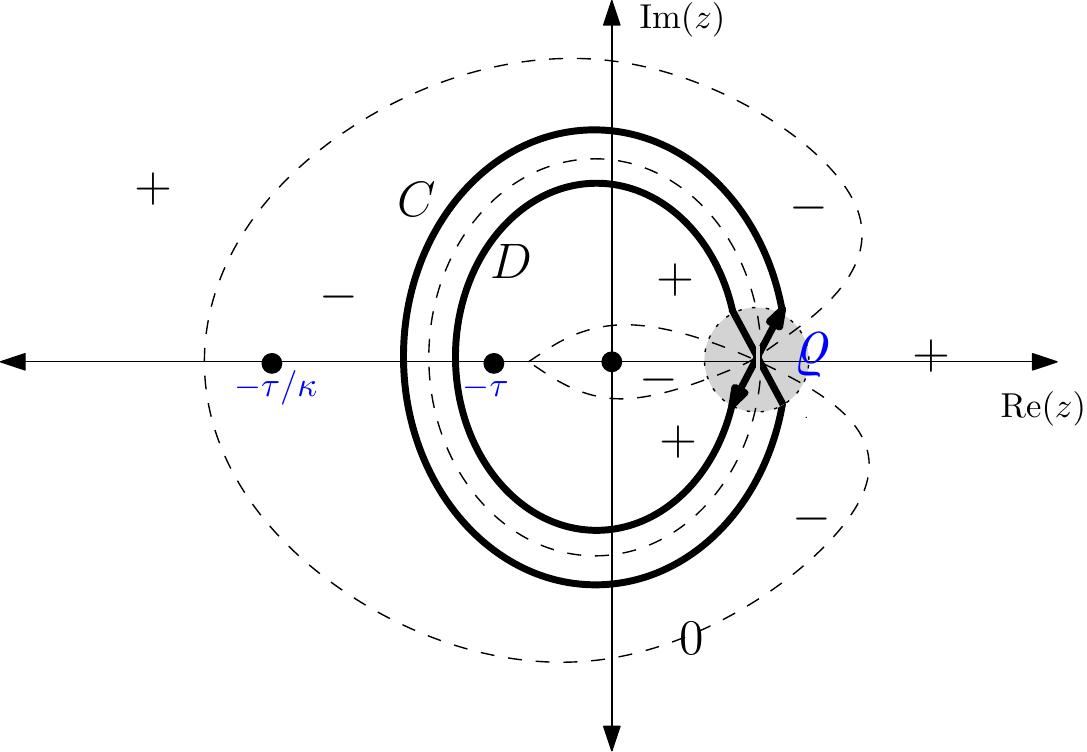}}}
 \caption{Zero contours for $\Re(G(z)) - \Re(G(\varrho))$
  are depicted by dashed lines (with the number 0 marked along them).
  Between these dashed lines lie areas of positive and negative real part, indicated by the occurrence of plus and minus signs.
  The contours $C$ and $D$ are depicted by solid lines.
 The region of the contours around the $\varrho$ are blown up on Figure \ref{Figure_Contours2} the right and coincide locally
with the limiting contours $\tilde{C}$ and $\tilde{D}$.
 \label{Figure_Contours}}
\end{center}
\end{figure}

We can now conclude about the features of the level lines (solutions) $\Re(G(z))=\Re(G(\varrho))$.
Namely, the level lines form $3$ closed loops as shown in Figure \ref{Figure_Contours}: all the
loops pass through $\varrho$, their second points of intersection with real axis lie on the
intervals $(-\infty,-\tau/\kappa)$, $(-\tau/\kappa,-\tau)$ and $(-\tau,0)$, respectively. The sign
of $\Re(G(z))-\Re(G(\varrho)$ alternates over the domains bounded by level lines, as shown in
Figure \ref{Figure_Contours}.

Finally, we claim that all the loops are ``star--shaped'', which means for each angle $\phi$ each
loop has precisely one point $z$ satisfying $Arg(z)=\phi$. To prove that, we fix $z_0=x_0+\ii y_0$
with $Arg(z_0)=\phi$ (we can assume that $0<\phi<\pi$ and  the case of real $z$ was studied
before) and consider the function of real variable $q$ given by
$$
\Re\bigl(G(z_0 q)\bigr)= \ln\big|1+ z_0 q \kappa \tau^{-1}\big| - \nu \ln\big|1+ z_0 q
\tau^{-1}\big| + m_{\nu} \ln|z_0 q|.
$$
Note that we allow $q$ to be negative here. We have
$$
\frac{\partial}{\partial q} \Re\bigl(G(z_0 q)\bigr)= \Re\left(\frac{z_0 \kappa\tau^{-1}}{1+ z_0 q
\kappa \tau^{-1}} - \frac{\nu z_0 \tau^{-1}}{1+ z_0 q \tau^{-1}} + \frac{m_{\nu}}{q} \right)
$$
The last formula implies that the roots of $\frac{\partial}{\partial q} \Re(G(z_0 q)=0$ are the
roots of a degree $4$ polynomial (in $q$) and, thus, there are at most $4$ of them. Taking into
the account the singularity at $0$ we conclude that for each $W\in\mathbb R$ the equation
$\Re(G(z_0 q))=W$ has at most $6$ solutions: Indeed, there are $m$ positive solutions and $n$
negative solutions, between each solution a root of the derivative should exist, thus,
$(m-1)+(n-1)\le 4$.

 Now if all the loops of $\Re(G(z))=\Re(G(\varrho))$ are star-shaped, then this already
gives precisely $6$ solutions; if one of the loops were not star--shaped, then we would have had
more solutions, which is impossible.

\bigskip

After having established the validity of Figure \ref{Figure_Contours} we proceed to the contour
deformations. We deform the $w$--contour and $v$--contour to curves $C$ and $D$, shown in Figure
\ref{Figure_Contours}, respectively. The $w$--contour $C$ goes through the critical point
$\varrho$ and departs it at angles $\pm \pi/3$ (oriented with increasing imaginary part).
Likewise, the $v$ contour $D$ goes through $\varrho - \varrho \sigma_{\nu}^{-1} L^{-1/3}$ and
departs at angles $\pm 2\pi/3$ (oriented with decreasing imaginary part -- as is a consequence of
the change of variables). The reason for the shift in the location of the $v$ contour is to avoid
the pole from the denominator $v-w'$. Outside a small neighborhood of $\varrho$ both contours
closely follow one of the loops of $\Re(G(z))=\Re(G(\varrho))$ --- the one which is between them
and then crosses the negative real axis between $-\tau/\kappa$ and $-\tau$. The $C$ contour stays
outside this loop, i.e.\ $\Re(G(w))-\Re(G(\varrho))<0$ along it (for $z\ne\varrho$), while the $D$
contour is inside the loop and $\Re(G(v))-\Re(G(\varrho))>0$ along it (outside a small
neighborhood of $\varrho$).

As a consequence, on the new contours $C$ and $D$, for any $\eps>0$ there exists $\delta>0$ such
that as long as either $w$ or $v$ is outside $\eps$--neighborhood of $\varrho$, we have
$\Re(G(w))-\Re(G(v))<\delta$. Therefore, the integrand in \eqref{eq_kernel_plugged_3} would decay
exponentially fast as $L\to\infty$. It follows (see Section \ref{section_proof_2}) that
asymptotically only $w$, $v$ in a small neighborhood of $\varrho$ influence the desired Fredholm
determinant.

\begin{figure}[h]
\begin{center}
 {\scalebox{1}{\includegraphics{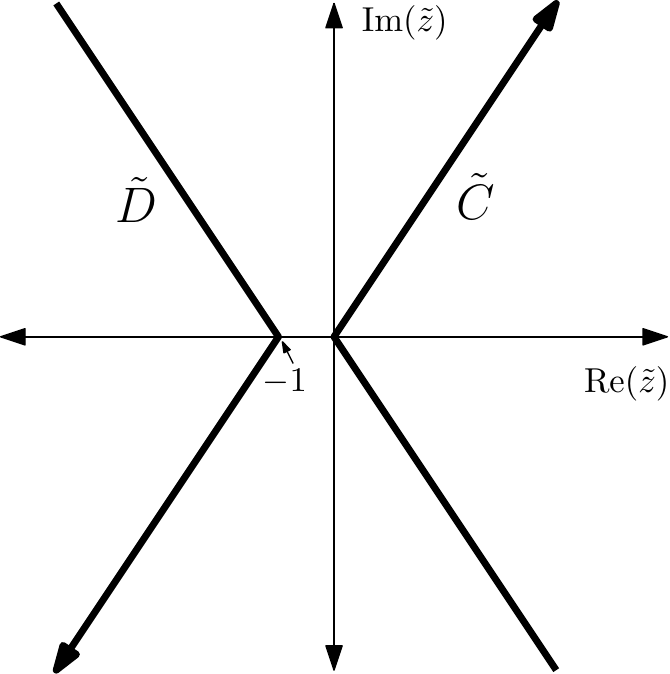}}}
 \caption{The contours $\tilde{C}$ and $\tilde{D}$ in $\tilde z$ plane  obtained from $C$ and $D$ in $z$ plane
 after the change of variables $z=\varrho + L^{-1/3} \varrho \sigma_\nu^{-1} \tilde z$. The variable
 $\tilde w$ is integrated over $\widetilde C$ and $\tilde v$
 is integrated over $\widetilde D$.}
 \label{Figure_Contours2}
\end{center}
\end{figure}

In a small neighborhood of $\varrho$ we make a change of variables $w=\varrho + L^{-1/3} \varrho
\sigma_\nu^{-1} \tilde w$, $v=\varrho + L^{-1/3} \varrho \sigma_\nu^{-1} \tilde v$. The contours
arising after this change of variables are shown in Figure \ref{Figure_Contours2}. Note that we
have integration both in $v$ (in \eqref{eq_kernel_plugged_3}) and in $w$ (in the definition of
Fredholm determinant), this means that we should also absorb in the kernel $K$ the factor
$L^{-2/3} \varrho^2\sigma_\nu^{-2}$ arising from the Jacobian of the change of coordinates. As a
result, using the Taylor expansion \eqref{eq_Taylor} and the expansion
$\ln(\varrho+z)=\ln(\varrho)+z/\varrho+o(z)$, the kernel in \eqref{eq_kernel_plugged_3} transforms
into
\begin{equation}
\label{eq_kernel_plugged_4}
 \tilde K^{b_1,b_2}_{\zeta(L)} (\tilde w,\tilde w')=\frac{L^{-1/3}
\varrho \sigma_\nu^{-1}}{\ln(\tau) 2\ii} \sum_{k=-\infty}^{\infty} \int_{\widetilde D}
 \frac{ \exp\left[ \frac{\tilde w^3}{3}-\frac{\tilde v^3}{3} + h(\tilde v-\tilde w) + o(1)\right]
 \cdot  d\tilde v}{(\tilde v-\tilde w')\sin\left(\frac{\pi}{\ln(\tau)} (L^{-1/3} \sigma_\nu^{-1}(\tilde v-\tilde w)
 +2\pi\ii k)\right) (\varrho + o(1)) }.
\end{equation}
When $k\ne 0$,
$$
 \lim_{L\to\infty} \sin\left(\frac{\pi}{\ln(\tau)} (L^{-1/3} \sigma_\nu^{-1}(\tilde v-\tilde w)
 +2\pi\ii k)\right)=\sin\left(\frac{\pi}{\ln(\tau)}2\pi\ii k\right)\ne 0,
$$
and the corresponding term in the summation of $k$ vanishes due to the $L^{-1/3}$ prefactor. On
the other hand, when $k=0$,  as $L\to\infty$
$$
 \sin\left(\frac{\pi}{\ln(\tau)} (L^{-1/3}
\sigma_\nu^{-1}(\tilde v-\tilde w)
 )\right)\approx \frac{\pi}{\ln(\tau)} L^{-1/3}
\sigma_\nu^{-1}(\tilde v-\tilde w)
$$
Plugging this  into \eqref{eq_kernel_plugged_4}, we conclude that
\begin{equation}
\label{eq_kernel_plugged_5} \lim_{L\to\infty} \tilde K^{b_1,b_2}_{\zeta(L)} (\tilde w,\tilde
w')=\frac{1}{2\pi\ii} \int_{\widetilde D}
 \frac{ \exp\left[ \frac{\tilde w^3}{3}-\frac{\tilde v^3}{3} + h(\tilde v-\tilde w) \right]
 \cdot  d\tilde v}{(\tilde v-\tilde w') (\tilde v-\tilde w)
  }.
\end{equation}
Denoting the right--side of \eqref{eq_kernel_plugged_5} as $K^{Ai}$, we conclude that
\begin{equation}
\label{eq_final_convergence}
  \lim_{L\to\infty}
  \det\left(1+K^{b_1,b_2}_{\zeta(L)}\right)_{C_r}=\det\left(1+K^{Ai}\right)_{\widetilde C}.
\end{equation}
The last determinant is (upon a change of variables $\tilde v\to -\tilde v$, $\tilde w\to-\tilde
w$, $\tilde w' \to -\tilde w'$) a standard expression for $F_{GUE}(h)$, cf.\ \cite{TW_determ},
\cite[Lemma 8.6]{BCF}.

\subsection{On the estimates}
\label{section_proof_2}

In the argument of Section \ref{section_proof_1} we were dealing with leading terms of the
asymptotics without making the estimates for the remainders. All such estimates are fairly
standard, let us only point where they are required and where analogous estimates can be found in
the literature.

\begin{enumerate}
 \item In order to justify formula \eqref{eq_kernel_plugged_4} and the following pointwise limit in it,
 we need to estimate the integral
 outside a small neighborhood of the critical point $\varrho$. This is a usual estimate of
 steepest descent method of the analysis of integrals, cf.\ \cite{Copson}, \cite{Er}. In the
 related context of directed polymers in random media a very similar justifications were done recently in
 \cite[Section 5.2]{BCF}, \cite[Section 2]{BCR}.

 \item The formula \eqref{eq_kernel_plugged_5} leads to the termwise limit for the Fredholm
 determinant of Proposition \ref{Proposition_analytical_converence}. To justify the $L\to\infty$
 limit for the sums, i.e.\ equality \eqref{eq_final_convergence} we need also certain uniform
 estimates for the remainder of the series (large $n$ terms of Definition \ref{Def_fredholm}).
 Using Hadamard's inequality this readily follows from our limit analysis of the kernel
 $K^{b_1,b_2}_{\zeta(L)}$ and we again refer to \cite[Section 5.2]{BCF}, \cite[Section 2]{BCR} and
 references therein for additional details.
\end{enumerate}

\section{Proofs of Theorem \ref{Theorem_LLN} and Theorem \ref{Theorem_fluctuations}}

Due to the identification between the the configurations of the six--vertex model and interacting
particle system $\x^{b_1,b_2}(t)$ explained in Section \ref{Section_P_as_interacting},
$H(x,y;\omega)$ of Theorem \ref{Theorem_fluctuations} is the same as $N_x(t)$ with $t=y$ of
Theorem \ref{Theorem_asymptotics}. Thus, these theorems are equivalent, and passing from one to
another is a matter of changing the notations.

Further, Theorem \ref{Theorem_fluctuations} readily implies Theorem \ref{Theorem_LLN} for $x,y$
satisfying $\frac{1-b_1}{1-b_2}<\frac{x}{y}<\frac{1-b_2}{1-b_1}$.

Let us study the remaining $x,y$. We start from $x\le \frac{1-b_1}{1-b_2} y$. Choose any $\eps>0$
and write
$$
 \limsup_{L\to\infty} \frac{H(Lx,Ly;\omega)}{L}\le  \limsup_{L\to\infty} \frac{H(L(\frac{1-b_1}{1-b_2} y+\eps),Ly;\omega)}{L}=
 \H\left(\frac{1-b_1}{1-b_2} y+\eps,y\right).
$$
The definition of $\H$ implies that for any $y$,
$$
 \lim_{\eps\to 0}  \H\left(\frac{1-b_1}{1-b_2} y+\eps,y\right)=0.
$$
Since $H(Lx,Ly;\omega)\ge 0$ we conclude that for $x,y$ satisfying  $\frac{x}{y}\le
\frac{1-b_1}{1-b_2}$ we have
$$
 \lim_{L\to\infty} \frac{H(Lx,Ly;\omega)}{L}=0.
$$
It remains to consider the case $x\ge \frac{1-b_2}{1-b_1} y$. For this note that for any  $x$,
$y>0$ we have (almost surely)
$$
 H(x,y;\omega)\ge L(x-y) -2.
$$
To prove this inequality observe that $H(x,0;\omega)=\lfloor x\rfloor$ and when we increase $y$ by
$1$ the height function decreases at most by $1$. Therefore,
\begin{equation}
\label{eq_x8}
 \liminf_{L\to\infty} \frac{H(Lx,Ly;\omega)}{L}\ge x-y.
\end{equation}

On the other hand, for $x\ge \frac{1-b_2}{1-b_1} y$ and any $\eps>0$ we have (again using the fact
that $H$ does not change by more than one as we move by one unit along the grid)
\begin{equation}
\label{eq_x9}
 H(Lx,Ly;\omega)\le H\left(L  \left(\frac{1-b_2}{1-b_1} y-\eps\right), Ly;\omega\right) + L\left(x- \frac{1-b_2}{1-b_1} y \right) +1.
\end{equation}
Using the definition of the limit $\H(x,y)$ we see that for any $y$
\begin{equation}
\label{eq_x10}
 \lim_{\eps\to 0} \H\left(\frac{1-b_2}{1-b_1} y-\eps, y\right)=  \frac{1-b_2}{1-b_1} y-y.
\end{equation}
Therefore, sending $L\to\infty$ in \eqref{eq_x9} we conclude that (in probability)
\begin{equation}
\label{eq_x11}
 \limsup_{L\to\infty} \frac{H(Lx,Ly;\omega)}{L}\le\left( \frac{1-b_2}{1-b_1} y-y\right)+ \left(x- \frac{1-b_2}{1-b_1}
 y\right)=x-y.
\end{equation}
Combining \eqref{eq_x8} and \eqref{eq_x11} we conclude that for $x\ge \frac{1-b_2}{1-b_1} y$,
$$
 \lim_{L\to\infty} \frac{H(Lx,Ly;\omega)}{L}=x-y,
$$
which finishes the proof of Theorem \ref{Theorem_LLN}.

\end{document}